\documentclass[12pt]{article}
\usepackage{amsmath,amssymb,mathrsfs,longtable,mathtools,hyperref,enumitem}
\usepackage{amsthm,rotating,color,epsfig,mdframed,url}
\usepackage[utf8]{inputenc}
\usepackage{graphicx} 
\graphicspath{ {pic/} }

\newtheorem{prop}{Proposition}[section]
\newtheorem{coro}[prop]{Corollary}
\newtheorem{lemm}[prop]{Lemma}
\newtheorem{clai}[prop]{Claim}
\newtheorem{theo}[prop]{Theorem}

\theoremstyle{definition}
\newtheorem{defi}[prop]{Definition}

\newtheorem{ques}[prop]{Question}

\newtheorem{rema}[prop]{Remark}

\begin{document}

\title{Schoenflies problem for area-preserving biLipschitz mappings}
\author{Maxim~Prasolov}
\date{}

\maketitle

\makeatletter
\renewcommand{\@makefnmark}{}
\makeatother

\begin{abstract}
We prove that any biLipschitz mapping of the boundary of the unit disk onto the boundary of the domain of the same area can be extended to a biLipschitz mapping of the whole plane which preserves the area of any measurable subset.
\end{abstract}

\tableofcontents

\section{Introduction}

Let $X,Y$ be metric spaces. A mapping $f:X\to Y$ is called $L$-biLipschitz if for any points $x,y\in X$
$$
\frac1L\cdot\mathrm{dist}(x,y) \le \mathrm{dist}(f(x),f(y)) \le L\cdot\mathrm{dist}(x,y).
$$
A mapping is called biLipschitz if it is $L$-biLipschitz for some $L.$

Let $X$ be a subset of Euclidean plane $\mathbb R^2$. We say that a mapping $f:X\to\mathbb R^2$ preserves the area if for any measurable subset $A$ of $X$
$$
\mathrm{Area} (A) = \mathrm{Area}(f(A)).
$$

Let $\Omega$ be an open subset of $\mathbb R^2.$ A mapping $f:\Omega\to\mathbb R^2$ is called locally piecewise affine if $\Omega$ is a locally finite union of triangles on which $f$ is affine.

Let $B = [-1;1]^2\subset\mathbb R^2.$

\begin{theo}\label{main-theorem}
For any $L\ge1$ there exist positive constants $K$ and $R,$ depending solely on $L,$ such that any $L$-biLipschitz mapping of $\partial B$ onto the boundary of a domain of area $4$ can be extended to an area-preserving $K$-biLipschitz mapping of $\mathbb R^2$ that is locally piecewise affine in $\mathbb R^2\setminus\partial B$ and coincides with a translation outside some disk of radius $R$.
\end{theo}

We do not provide any explicit bound on $K$ or $R$ in terms of $L$.



The author is interested in Theorem~\ref{main-theorem} due to its applications in contact topology. Suppose that a diffeomorphism $\varphi:\mathbb R^2\to\mathbb R^2$ preserves area. Then $\varphi^{*}(xdy) - xdy$ is a closed form thus there exists a function $f:\mathbb R^2\to \mathbb R$ such that $df = \varphi^{*}(xdy) - xdy.$ Then a mapping

$$
(x,y,z) \mapsto (\varphi(x,y), z - f(x,y))
$$
preserves the 1-form $dz + xdy.$ A 1-form $\alpha$ on a 3-manifold is called contact if $\alpha\wedge d\alpha\neq 0$ everywhere, so $dz + xdy$ is a contact form on $\mathbb R^3.$ The distribution $\ker \alpha$ is called a contact structure. Diffeomorphisms preserving the contact structure are called contact or contactomorphisms. Generalizing this example to biLipschitz setting L.\,Capogna and P.\,Tang proved in~\cite[Theorem 5.3]{CapoTang} that such lifting of a biLipschitz homeomorphism is quasiconformal for Carnot-Carath\'eodory metric on the Heisenberg group. Using the same lifting in Section~\ref{contact-section} we construct a standard tubular neighborhood of any Legendrian Lavrentiev link.

\begin{coro}
\label{contact-corollary}
Let $M$ be a contact 3-manifold, $\iota:\mathbb S^1\to M$ be a biLipschitz parametrization of a Legendrian Lavrentiev curve $L = \iota(\mathbb S^1).$ Then there exist $\delta>0,$ a neighborhood $U$ of $L$ and a contact locally biLipschitz homeomorphism $\Phi:(-\delta;\delta)\times\mathbb S^1\times\mathbb R\to U$ such that $\Phi(0,y,0) = \iota(y)$ for any $y\in\mathbb S^1,$ where the manifold $(-\delta;\delta)\times\mathbb S^1\times\mathbb R$ is equipped with a contact structure $\ker(dz + xdy).$
\end{coro}

In~\cite{Tuk} P.\,Tukia proved that any biLipschitz mapping from $\partial B$ to $\mathbb R^2$ can be extended to a biLipschitz self-homeomorphism of the whole plane (see~\cite{DanPrat, JerKen, Tuk2} and~\cite[Theorem 7.9]{Pomm} for other proofs). So to prove Theorem~\ref{main-theorem} it is sufficient to construct a biLipschitz homeomorphism of the plane which is the identity on $\partial B$ and whose Jacobian is the same as the Jacobian of a map constructed by Tukia. Essentially we follow this approach but there is an obstacle. If $f$ is a Jacobian of a biLipschitz mapping, then $f\in L^{\infty}(\mathbb R^2)$ and $1/f \in L^{\infty}(\mathbb R^2).$ For instance, if $\Phi(x,y) = (2x, y)$ if $x\ge 0$ and $\Phi(x,y) = (x,y)$ if $x<0$, then $\Phi$ is biLipschitz and $\det D\Phi=2$ if $x>0$ and $\det D\Phi=1$ if $x<0.$ However in~\cite{BurKlei} D.\,Burago and B.\,Kleiner constructed a continuous function on the square with values in $[1;1+\varepsilon]$ with arbitrarily small $\varepsilon > 0$ which can not be a Jacobian of a biLipschitz mapping. They and independently C.\,McMullen in \cite{McM} also obtained a similar result with a function having only two values.

We use an important feature of Tukia's mapping. Consider a tiling of the square $B$ into quadrilaterals as in Figure~\ref{tiling}. Images of these quadrilaterals under his mapping constitute a finite set of similarity classes of polygons, see Theorem~\ref{Tukia-theorem}. This fact allows us to modify Tukia's mapping on each element of the tiling preserving the mapping on the boundary of the quadrilateral so that the new mapping has constant Jacobian on the quadrilateral. For functions which are constant on each element of the tiling we were able to solve the prescribed Jacobian problem.

\begin{prop}
\label{jacobian-problem}
Let a square be dissected into quadrilaterals as in Figure~\ref{tiling}. Then for any positive $C$ there exists $L$ such that for any function $f$ on the square such that
\begin{itemize}
\item $f$ is constant on each quadrilateral,
\item $1/C \le f \le C$ everywhere,
\item the integral of $f$ over the square equals the area of the square
\end{itemize}
there exists a locally affine $L$-biLipschitz mapping from the square to itself whose Jacobian equals $f$ almost everywhere.
\end{prop}



\begin{figure}[h]
\center
\includegraphics{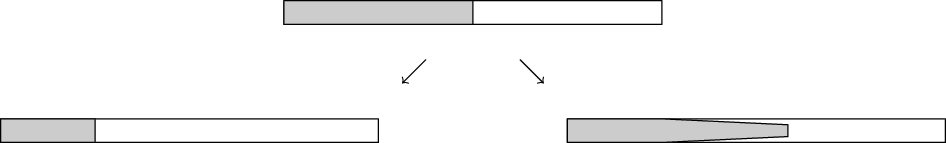}
\caption{Images of halves of the rectangle under a 2-biLipschitz mapping on the left and under an area-preserving mapping on the right. Two mappings coincide on the boundary.}
\label{counterexample}
\end{figure}

When one defines a biLipschitz mapping first on some 1-dimensional subcomplex on the plane and then on each component of its complement requiring that the Jacobian is constant, it is important that pieces are not thin. For instance, consider a rectangle $[0;1]\times[0;h]$ and its self-mapping $x\mapsto \frac{1}2x$ if $x\le \frac12$ and $x\mapsto \frac32x -\frac12$ if $x\ge\frac12,$ $y\mapsto y$ (Figure~\ref{counterexample}). This mapping is 2-biLipschitz. Consider the restriction of this mapping to the boundary. Every its area-preserving extension maps a middle vertical line to a curve of length greater than $\frac12,$ therefore the extension can not be $1/2h$-biLipschitz. One can also check that any $L$-biLipschitz self-homeomorphism of the boundary of the rectangle can be extended to a $\widetilde L$-biLipschitz self-homeomorphism of the rectangle, where the constant $\widetilde L$ depends only on $L.$

This issue arose when the author tried to adopt a proof by S.\,Daneri and A.\,Pratelli (see~\cite{DanPrat}). If the given mapping $\phi:\partial B\to\mathbb R^2$ is piecewise linear they construct a piecewise affine extension $B\to\mathbb R^2.$ They use a triangulation of the polygon bounded by $\phi(\partial B)$ such that any vertex of the triangulation lies on the boundary of the polygon. But for any $L$ one can construct an $L$-biLipschitz piecewise linear mapping $\phi$ such that any such triangulation contains arbitrarily thin triangles (for example, if some concave arc in $\phi(\partial B)$ consists of sufficiently small edges). Thin triangles make hard to apply the method of~\cite{DanPrat} to constructing area-preserving extensions.


We also prove two relative versions of Theorem~\ref{main-theorem}. Let $r\cdot B$ denote a square $[-r;r]^2\subset\mathbb R^2.$

\begin{theo}
\label{annulus-version}
For any $L,$ $r_0$ and $r_1$ such that $0<r_0<r_1<1$ there exist $R = R(L)$ and $K = K(L, r_0, r_1)$ such that any $L$-biLipschitz mapping $\phi:\partial B\to\mathbb R^2$ such that $\phi(\partial B)$ bounds a domain of total area $4$ containing $r_1\cdot B$
can be extended to locally piecewise affine in the complement to $\partial B$ area-preserving  $K$-biLipschitz mapping $\Phi:\mathbb R^2\to\mathbb R^2$ such that $\Phi(z) = z$ if $|z|\ge R$ or $|z|\le r_0.$ 
\end{theo}

In Theorem~\ref{annulus-version} we can't expect that $K$ depends only on $L$ for the reason similar to the one considered in the example with a thin rectangle in Figure~\ref{counterexample}. 

\begin{theo}
\label{the-technical-lemma-without-parameters}
For any real numbers $L,$ $r_0,$ $r_1,$ $R_0,$ $R_1$ such that $0 < r_0 < r_1 < 1 <R_1<R_0$  there exists $K$ such that any $L$-biLipschitz mapping $\phi:\partial B\to\mathbb R^2$ such that 
\begin{itemize}
\item $\phi(\partial B)$ bounds a domain of total area $4$ containing $r_1\cdot B;$
\item $\phi(\partial B)\subset R_1\cdot B$
\end{itemize}
can be extended to a mapping $\Phi:\mathbb R^2\to\mathbb R^2$ such that
\begin{itemize} 
\item it is locally piecewise affine in the complement to $\partial B$, area-preserving and $K$-biLipschitz;
\item $\Phi(z) = z$ if $|z|\ge R_0$ or $|z|\le r_0.$ 
\end{itemize}
\end{theo}

\subsection{General plan of the proofs}

We prove Corollary~\ref{contact-corollary} in Section~\ref{contact-section} relying on Theorem~\ref{the-technical-lemma-without-parameters}. In Sections~\ref{auxiliary-section},~\ref{main-section} and~\ref{main-lemma-section} we prove Theorems~\ref{main-theorem}, \ref{annulus-version} and \ref{the-technical-lemma-without-parameters} in parallel. 

To prove Theorem~\ref{main-theorem}, we first extend the mapping to the interior of $B$ in Proposition~\ref{square-extension-lemma}, and then to the exterior in Proposition~\ref{extension-to-complement-lemma}. In Lemma~\ref{Tukia-theorem} we review Tukia's construction of a biLipschitz extension over the square without the area-preserving property. This extension comes with a tiling of $B$ into quadrilaterals (see Figure~\ref{tiling}) and in fact we can assume that the extension has a constant Jacobian on each quadrilateral. We then construct a mapping from $B$ to itself with the same Jacobian that restricts to the identity on the boundary. This is the statement of Proposition~\ref{jacobian-problem}, it is proved right before Lemma~\ref{areafication}. The composition of this mapping with Tukia's extension is the desired area-preserving extension (see Lemma~\ref{areafication} for details). To construct a mapping with a prescribed Jacobian, we first specify the image of the tiling in Lemma~\ref{half-theorem}  by prescribing the area of the image of each quadrilateral. The main difficulty lies in preventing these images from becoming too thin. This is the core contribution of the paper, and we isolate the heart of this construction in Section~\ref{main-lemma-section}. Once the image of the tiling is constructed, we define the mapping on the interiors of the quadrilaterals. For this, we apply Lemma~\ref{square-to-graph2}, where we construct a constant-Jacobian mapping from the square to the region between the graphs of two Lipschitz functions. The core of this lemma is established in Lemma~\ref{pl-square-to-graph0}, which in turn relies on Lemmas~\ref{skew-lemma} and~\ref{trapezoid-lemma}, as well as Proposition~\ref{from-line-to-plane}.

We then extend the mapping to the exterior of the square. Up to a translation, we may assume the origin lies inside the bounded component of the complement to the image of $\partial B.$ In Section~\ref{extension-to-complement}, we use area-preserving, orientation-reversing, biLipschitz homeomorphisms to exchange certain interior and exterior regions near the boundaries of the square and its image. This allows us to apply Proposition~\ref{square-extension-lemma} (the interior extension), provided the extension is the identity on a neighborhood of the origin. We achieve this condition in Proposition~\ref{extension-to-annulus-lemma} within Section~\ref{correction-on-some-square}.

The proofs of Theorems~\ref{annulus-version} and~\ref{the-technical-lemma-without-parameters} are similar to that of Theorem~\ref{main-theorem}. The only difference is that the extension must restrict to the identity on all of $B$ except for a small neighborhood of $\partial B.$ To achieve this, we revisit Tukia's construction in Section~\ref{merge-tiling}. Every quadrilateral in the tiling is parametrized by a square; under this parametrization, the restriction of a mapping on a quadrilateral can be viewed as a mapping from a square to the plane. A key feature of Tukia's extension is that the set of similarity classes of these restrictions is finite, with its cardinality depending only on the biLipschitz constant of the initial boundary mapping of $\partial B.$ In Lemma~\ref{merging-elements-of-tiling-lemma}, we modify the tiling so that one of quadrilaterals covers the entire interior of $B$ up to a small neighborhood of $\partial B.$ The finiteness of the number of similarity classes then allows us to modify Tukia's extension so that it restricts to the identity outside this neighborhood of $\partial B$ keeping control over the biLipschitz constant. We conclude the proofs of the main theorems at the end of Section~\ref{main-section}.

\section*{Acknowledgments}
I'm grateful to Sergei Kuksin, Valerii Pchelintsev and Alexander Tyulenev for the confirmation that the result is new. I would like to thank Aleksandr Berdnikov for discussing the prescribed Jacobian problem and for pointing out some inaccuracies in the preliminary version of the paper. Special thanks to Ivan Dynnikov and Sergei Kuksin for continuous support and for the interest to this work.
I am indebted to the referee for the careful reading and many suggestions on improving the presentation of the text.

\section{Some auxiliary area-preserving piecewise affine mappings}
\label{auxiliary-section}
Let $X, Y$ be metric spaces. A mapping $f:X\to Y$ is called {\it $L$-Lipschitz} if for any $x,y\in X$
$$
\mathrm{dist}(f(x), f(y)) \le L\cdot\mathrm{dist}(x,y).
$$

By {\it a Lipschitz constant} of $f$ we call
$$\mathrm{lip}(f) = \sup\limits_{x\neq y\in X}\frac{\mathrm{dist}(f(x),f(y))}{\mathrm{dist}(x,y)}.$$

The real number
$$\mathrm{bilip}(f) = \max\left(\mathrm{lip}(f), \mathrm{lip}\left(f^{-1}\right)\right)$$ 
is called {\it a biLipschitz constant} of~$f$.

A mapping $f:X\to Y$ is called locally ($L$-)(bi)Lipschitz, if there exists an open cover $X = \cup_{\alpha} U_{\alpha}$ such that for any $\alpha$ the restriction $f\big|_{U_{\alpha}}$ is ($L$-)(bi)Lipschitz.

Let $\Omega$ be open in $\mathbb R^n.$ By Rademacher Theorem any locally Lipschitz mapping $f:\Omega\to\mathbb R^m$ is differentiable almost everywhere. Its differential $Df$ coincides almost everywhere with some measurable bounded mapping $\Omega \to \mathrm{Mat}_{m\times n}$ (see~\cite[Section 4]{Hein}). If $M$ is an essential supremum of the operator norm of the matrix $Df$, the mapping $f$ is locally $M$-Lipschitz, where the operator norm of a matrix $A$ is
$$||A|| = \sup\limits_{|v|=1}|Av|.$$

To estimate the operator norm we will use Euclidean norm
$$|A| = \sqrt{\mathrm{tr} \left(A^T A\right)}$$
and an inequality $||A|| \le |A|$ (the inequality is strict if $A$ is not degenerate).

A domain $\Omega\subset \mathbb R^n$ is called {\it $C$-quasiconvex} if for any two points $x,y\in\Omega$ there exists a path connecting them whose length is not greater than $C\cdot\mathrm{dist}(x,y).$

If a mapping $f:\Omega\to\mathbb R^m$ is locally $L$-Lipschitz and the domain $\Omega$ is convex, $f$ is $L$-Lipschitz. If $\Omega$ is not convex, we will use the following statement to estimate the Lipschitz constant (see~\cite[Lemma 2.2]{Hein}):

\begin{lemm}
Let $f:\Omega\to\mathbb R^m$ be locally $L$-Lipschitz and $\Omega$ be $C$-quasiconvex. Then $f$ is $CL$-Lipschitz.
\end{lemm}

From now we assume that $n=m=2.$ If $f$ is biLipschitz and $\det Df = 1$ almost everywhere, then $f$ preserves area. If $\det A = 1$ and $||A||\le L,$ then $||A^{-1}|| \le L.$ Indeed, let $v,w$ be an orthonormal basis, then $1 = \mathrm{Area}(Av, Aw) \le |Av|\cdot|Aw| \le |Av|\cdot L.$ So $|Av|\ge L^{-1}.$ Therefore if an area-preserving mapping is $L$-Lipschitz, then it is $L$-biLipschitz.

\begin{lemm}
\label{skew-lemma}
Let $f(x)$ be an $L$-Lipschitz function defined on a connected subset $I$ of $\mathbb R$. Then a mapping
$$
\Phi: (x,y) \mapsto (x, y + f(x))
$$
defined on $\Omega = I\times\mathbb R$ is $\sqrt{2+L^2}$-biLipschitz and preserves area. If $f$ is piecewise linear, $\Phi$ is piecewise affine.
\end{lemm}
\begin{proof}
The Jacobi matrix of $\Phi$  is
$$
D\Phi = 
\begin{pmatrix}
1 & 0\\
f'(x) & 1
\end{pmatrix},
$$
hence $||D\Phi|| < |D\Phi| = \sqrt{2 + \left(f'\right)^2} \le \sqrt{2+L^2}$ and $\det D\Phi = 1.$
\end{proof}

\begin{figure}[h]
\center
\includegraphics{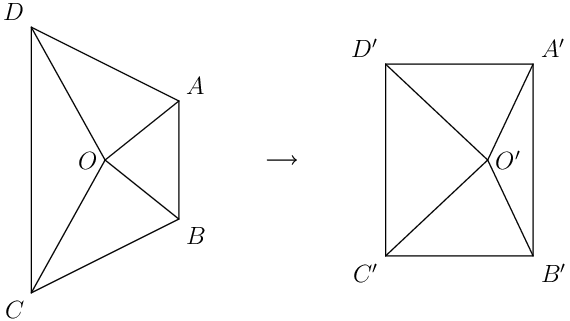}
\caption{The image of a trapezoid under a piecewise affine area-preserving mapping is a rectangle with the same height.}
\label{trapezoid-to-rectangle}
\end{figure}

Recall that the {\it height} of a trapezoid is the distance between the parallel sides.

\begin{lemm}
\label{trapezoid-lemma}
There exists a function $L(x,y)$ monotonic on each argument such that if
\begin{enumerate}[label=\arabic*)]
\item
$ABCD$ is a trapezoid with bases $AB$ and $CD,$ $AB = a,$ $CD = b,$ and the height of the trapezoid equals $h$;
\item $a<b,$ $2h \le a+b$ and the angles on the base $CD$ are not obtuse;
\item
the area of a rectangle $A'B'C'D'$ equals the area of the trapezoid and $B'C' = h,$
\end{enumerate}
then there exists an area-preserving piecewise affine $L((b-a)/h, b/a)$-biLipschitz mapping from the trapezoid to the rectangle which maps the vertices $A,B,C,D$ respectively to the vertices $A',B',C',D'$ and whose restriction on each side of the trapezoid is affine.
\end{lemm}
\begin{proof}
We assume that the bases are parallel to $y$-axis. Apply a mapping from Lemma~\ref{skew-lemma} to make the trapezoid isosceles. A function $f$ from Lemma~\ref{skew-lemma} is $\frac{b-a}{2h}$-Lipschitz because the angles on the longer base are not obtuse. Further we will assume that the trapezoid is isosceles from the beginning.

The areas of the trapezoid and the rectangle are equal so $A'B' = \frac{a+b}2.$

Let $M$ and $N$ be the middle points of the bases $AB$ and $CD$ respectively, $O$ be the middle point of $MN$. Let $M'$ and $N'$ be the middle points of the sides $A'B'$ and $C'D'$ respectively, $O'$ be the point on the segment $M'N'$ dividing it in the ratio $a:b$ (Figure~\ref{trapezoid-to-rectangle}). Then we compute the area of the following triangles:
$$
S_{\triangle OAB} = \frac12\cdot\frac{h}2\cdot a,\ 
S_{\triangle OCD} = \frac12\cdot\frac{h}2\cdot b,\ 
S_{\triangle O'A'B'} = \frac12h\cdot\frac{a}{a+b}\cdot\frac{a+b}2,\ 
S_{\triangle O'C'D'} = \frac12h\cdot\frac{b}{a+b}\cdot\frac{a+b}2.
$$

We see that $S_{\triangle OAB} = S_{\triangle O'A'B'}$ and $S_{\triangle OCD} = S_{\triangle O'C'D'}.$ Since the areas of the trapezoid and of the rectangle are equal, it follows that
$S_{\triangle OAD} = S_{\triangle O'A'D'}$ and $S_{\triangle OBC} = S_{\triangle O'B'C'}.$

Consider a mapping $\Phi$ which maps the points $O,A,B,C,D$ respectively to the points $O', A',B',$ $C',$ $D'$ and whose restriction to the triangles $\triangle OAB, \triangle OBC, \triangle OCD, \triangle ODA$ is affine.

Since the areas of the corresponding triangles are equal, the mapping preserves area.

We will assume that $O$ is the origin of the coordinate system. Let the axis $x$ be directed from the longer base of the trapezoid to the shorter one, and $y(A)>0.$ We may assume that the center of the rectangle is $O,$ the side $A'B'$ is parallel to the $y$-axis, $x(A')>0$ and $y(A')>0.$

Then on the triangle $\triangle OAB$ the Jacobi matrix equals $\begin{pmatrix}\frac{2a}{a+b} & 0\\ 0& \frac{a+b}{2a}\end{pmatrix}.$ So the mapping is $\frac{a+b}{2a}$-biLipschitz.

Similarly on the triangle $\triangle OCD$ the mapping is $\frac{2b}{a+b}$-biLipschitz.

We compute the Jacobi matrix on the triangle $\triangle OAD.$ 

$$
\overrightarrow{OA} = \left(\frac{h}2, \frac{a}2\right),\ 
\overrightarrow{OD} = \left(-\frac{h}2, \frac{b}2\right),\ 
\overrightarrow{O'A'} = \left(h\cdot\frac{a}{a+b}, \frac{a+b}4\right),\ 
\overrightarrow{O'D'} = \left(-h\cdot\frac{b}{a+b}, \frac{a+b}4\right),
$$

\begin{multline*}
D\Phi = 
\begin{pmatrix}
h\cdot\frac{a}{a+b} & -h\cdot\frac{b}{a+b}\\
\frac{a+b}4 & \frac{a+b}4
\end{pmatrix}
\begin{pmatrix}
\frac{h}2 & -\frac{h}2\\
\frac{a}2 & \frac{b}2
\end{pmatrix}^{-1}
=
\frac{2}{h(a+b)}
\begin{pmatrix}
h\cdot\frac{a}{a+b} & -h\cdot\frac{b}{a+b}\\
\frac{a+b}4 & \frac{a+b}4
\end{pmatrix}
\begin{pmatrix}
b & h\\
-a & h
\end{pmatrix}
= \\ =
\frac{2}{h(a+b)}
\begin{pmatrix}
\frac{2hab}{a+b} & h^2\cdot\frac{a-b}{a+b}\\
\frac{b^2-a^2}4 & h\cdot\frac{a+b}2
\end{pmatrix}
=
\begin{pmatrix}
\frac{4ab}{(a+b)^2} & 2h\cdot\frac{a-b}{(a+b)^2}\\
\frac{b-a}{2h} & 1
\end{pmatrix}.
\end{multline*}

$$
||D\Phi||^2 < |D\Phi|^2 = \left(\frac{4ab}{(a+b)^2}\right)^2 + \left(\frac{b-a}{2h}\right)^2 + \left(2h\cdot\frac{a-b}{(a+b)^2}\right)^2 + 1 \le 3 + \left(\frac{b-a}{2h}\right)^2.
$$

So the biLipschitz constant of the mapping $\Phi$ is at most
$$
\max\left(
2\cdot\frac{b}{a},
\sqrt{3 + \left(\frac{b-a}{2h}\right)^2}
\right).
$$

\end{proof}

\begin{defi}
If a trapezoid satisfies conditions 1) and 2) of Lemma~\ref{trapezoid-lemma}, we call a pair $((b-a)/h, b/a)$ a {\it distortion} of the trapezoid. We say that a pair $(x_0, y_0)$ is not greater than a pair $(x_1, y_1)$ if $x_0\le x_1$ and $y_0\le y_1.$
\end{defi}

\begin{figure}[h]
\center
\includegraphics{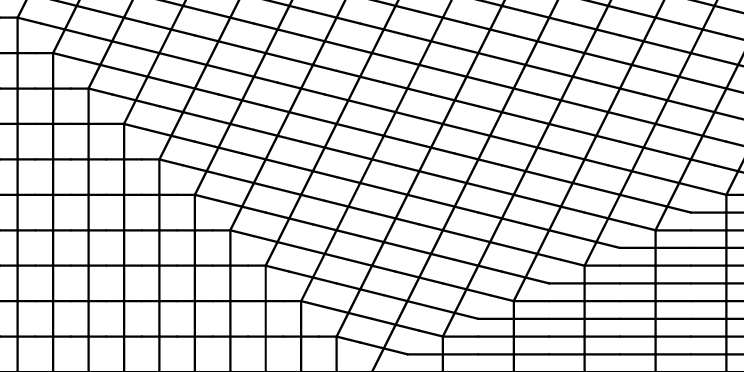}
\caption{The image of the coordinate upper half-grid under an area-preserving mapping. The restriction of the mapping on the $y=0:$ $x \mapsto 2x$ if $x\ge 0$ and $x\mapsto x$ if $x<0$.}
\label{line-to-line-figure}
\end{figure}

Let $\phi,\psi:\mathbb R\to\mathbb R$ be $L$-Lipschitz mappings such that $\phi\circ\psi = \mathrm{id}.$ We define functions

\begin{equation}
\label{line-to-line}
x(\widetilde x, y) = \frac{\phi(\widetilde x + y) + \phi(\widetilde x - y)}2,\quad
\widetilde y(\widetilde x, y) =  \frac{\phi(\widetilde x + y) - \phi(\widetilde x - y)}2.
\end{equation}

Since $\phi$ is monotonic, the function $x(\widetilde x, y)$ is monotonic on $\widetilde x.$ Therefore we can define the inverse function $\widetilde x(x, y).$ Then we define a function $\widetilde y (x,y) = \widetilde y(\widetilde x(x, y), y)$ (see an example in Figure~\ref{line-to-line-figure}).

\begin{prop}
\label{from-line-to-plane}
The mapping $\Psi:(x, y)\mapsto (\widetilde x(x, y), \widetilde y(x, y))$ is $2L^3$-biLipschitz area-preserving self-homeomorphism of the plane such that $\Psi(x, 0) = \psi(x)$ for any $x.$ If $\psi$ is piecewise linear, $\Psi$ is piecewise affine.
\end{prop}

\begin{proof}
The equality $\Psi(x, 0) = \psi(x)$ follows from the fact that $x(\widetilde x, 0) = \phi(x)$ and $\phi\circ\psi = \mathrm{id}.$

Now we show that $\Psi$ is a bijection. We fix a point $(\widetilde x_0, \widetilde y_0)$ and consider the equation $\widetilde y_0 = \widetilde y(\widetilde x_0, y)$ on $y.$ It has a unique solution because $\phi$ is strictly monotonic continuous and $\lim\limits_{t\to\pm\infty} \phi(t)= \pm\infty$. Let $y_0$ be the solution. Then $(\widetilde x_0, \widetilde y_0)\mapsto (x_0, y_0)$ is the inverse mapping, where $x_0 = x(\widetilde x_0, y_0).$

Then we compute $D\Psi$:

$$
\begin{aligned}
\frac{\partial \widetilde x}{\partial x} =& \left(\frac{\partial x}{\partial \widetilde x}\right)^{-1} = \frac2{\phi'(\widetilde x + y) + \phi'(\widetilde x - y)},
\\
\frac{\partial \widetilde x}{\partial y} =& -\frac{\partial x (\widetilde x, y)}{\partial y}\cdot\left(\frac{\partial x (\widetilde x, y)}{\partial \widetilde x} \right)^{-1} = - \frac{\phi'(\widetilde x + y) - \phi'(\widetilde x - y)}{\phi'(\widetilde x + y) + \phi'(\widetilde x - y)},\\
\frac{\partial \widetilde y}{\partial x} =& \frac{\partial \widetilde y(\widetilde x, y)}{\partial \widetilde x}\cdot \frac{\partial \widetilde x}{\partial x} = \frac{\phi'(\widetilde x + y) - \phi'(\widetilde x - y)}{\phi'(\widetilde x + y) + \phi'(\widetilde x - y)},\\
\frac{\partial \widetilde y}{\partial y} =& \frac{\partial \widetilde y(\widetilde x, y)}{\partial \widetilde x}\cdot \frac{\partial \widetilde x}{\partial y} + \frac{\partial \widetilde y(\widetilde x, y)}{\partial y}=
-\frac12 \frac{\left(\phi'(\widetilde x + y) - \phi'(\widetilde x - y)\right)^2}{\phi'(\widetilde x + y) + \phi'(\widetilde x - y)} + \frac{\phi'(\widetilde x + y) + \phi'(\widetilde x - y)}2 = \\
=& \frac{2\phi'(\widetilde x + y)\phi'(\widetilde x - y)}{\phi'(\widetilde x + y) + \phi'(\widetilde x - y)}.
\end{aligned}
$$

Since $\phi$ is $L$-biLipschitz, $L^{-1} \le |\phi'|\le L$ and $\phi'$ is of the same sign almost everywhere. Therefore
$$
||D\Psi||^2 < |D\Psi|^2 \le L^2 + L^4 + L^4 + L^6 \le \left(2L^3\right)^2,
$$

So $\Psi$ is locally $2 L^3$-biLipschitz. Since the image is the whole plane, so it is convex, the mapping is $2L^3$-biLipschitz.

It is clear that $\det D\Psi = 1.$
\end{proof}

\begin{ques}
Is it possible to extend any $L$-biLipschitz self-mapping of the line to an area-preserving $\left(M\cdot L\right)$-biLipschitz mapping of the plane, where $M$ is some constant not depending on $L$ and on the mapping?
\end{ques}
\begin{proof}[Answer]
Yes, Aleksandr Berdnikov constructed an $L$-biLipschitz area-preserving extension in the following way. Suppose that $\phi$ is smooth and increasing. Then there exists a time dependent vector field $v(x,t),$ $t\in[0;1]$, on $\mathbb R$ such that its flow at the moment $t=1$ equals $\phi$ and $|\mathrm{div}\, v|\le\ln\mathrm{bilip}(\phi)$ for any $t\in[0;1].$ For instance, we can take $v(\phi_t(x),t) = \frac{\partial}{\partial t}\phi_t (x),$ where $\phi_t(x_0) = \phi(0) + \int_0^{x_0} (\phi'(x))^t dx.$ This vector field can be extended to the whole plane by the formula
$$
v(x,y,t) = \left(\frac{v(x+y,t) + v(x-y,t)}2, \frac{-v(x+y,t) + v(x-y,t)}2 \right).
$$

It is a direct check that $\mathrm{div}\, v = 0$ (as of a vector field on the plane), $\left|\left|\nabla v\right|\right|\le\ln\mathrm{bilip}(\phi)$ and its flow at the moment $t=1$ extends $\phi.$ The case of an arbitrary $\phi$ can be settled by taking a limit.
\end{proof}

\begin{rema}
Since $\frac{\partial \widetilde x}{\partial y} = -\frac{\partial \widetilde y}{\partial x}$, the mapping $\Psi$ from Proposition~\ref{from-line-to-plane} is of the form $(x,y)\mapsto (u_x, -u_y)$, where $u$ is a weak solution to Cauchy problem $u_x\big|_{y=0} = \psi,$ $u_y\big|_{y=0} = 0,$ $u(0,0) = 0$ for the hyperbolic Monge--Amp\`ere equation $u_{xx} u_{yy} - u_{xy}^2 = -1.$ Formula~(\ref{line-to-line}) is a partial case of (4.15) in~\cite{Tun}.
\end{rema}

\begin{lemm}
\label{pl-square-to-graph0}
Let $f$ be a piecewise linear $L$-Lipschitz function on a segment $[0; 1]$. Let $C$ be a positive constant such that $C^{-1} \le f(x)\le C$ for any $x\in[0;1].$  Let $S = \int_0^1 f(x) dx.$ Then there exists a $K$-biLipschitz piecewise affine area-preserving mapping $\Phi: [0;1]\times [0;S] \to \mathbb R^2,$ where $K$ depends only on $L$ and $C$, such that
\begin{enumerate}[label = \arabic*),ref = \arabic*)]
\item
\label{top-bottom-sides-cond}
$\Phi(x, 0) = (x, 0)$ and $\Phi(x, S) = (x, f(x))$ for any $x\in[0;1]$;
\item
$\Phi(0, y) = \left(0, y\cdot\frac{f(0)}{S}\right)$ and $\Phi(1, y) = \left(1, y\cdot\frac{f(1)}{S}\right)$ for any $y\in[0;S].$
\end{enumerate}
\end{lemm}

\begin{proof} 
We construct the mapping in four steps.

\smallskip
\noindent{\bf Step 1.} First we construct a mapping $\Phi_0$, which satisfies all conditions of the lemma except condition~\ref{top-bottom-sides-cond}.

Let $0 = x_0 < x_1 < \dots < x_n = 1$ and the function $f$ be linear on each segment $[x_k; x_{k+1}].$ We may assume that $x_{k+1} - x_k < C^{-1}$ for any $k.$

\begin{figure}[h]
\center
\includegraphics{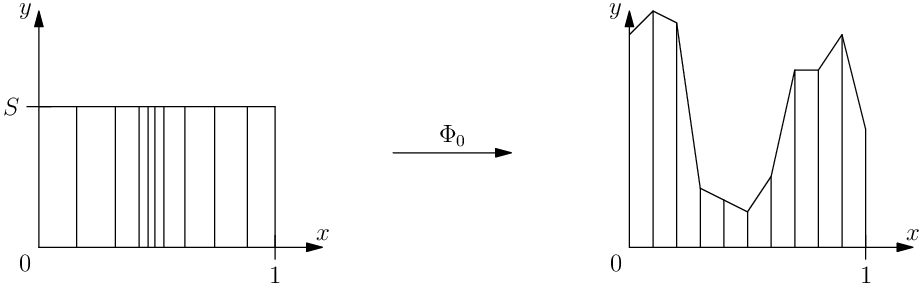}
\caption{Each rectangle is mapped to a trapezoid of the same area.}
\label{graph-step1}
\end{figure}

We dissect the rectangle $[0;1]\times[0;S]$ along vertical lines into rectangles $R_k$ which have sides $\int_{x_k}^{x_{k+1}}f(x)dx / S$ and $S$ (Figure~\ref{graph-step1}). A rectangle $R_k$ under an affine mapping can be mapped to a rectangle with sides $x_{k+1}-x_k$ and $(f(x_k)+f(x_{k+1}))/2$ which by Lemma~\ref{trapezoid-lemma} can be mapped to a trapezoid
$$T_k = \{(x,y):\ x\in[x_k; x_{k+1}],\ 0\le y\le f(x)\}.$$

We denote the composition of these two mapping by $\Phi_0.$ The biLipschitz constant of the affine mapping is at most $C^2.$ Trapezoids are of distortion at most $(L, C^2).$ Hence $\Phi_0$ is locally $L'(L,C)$-biLipschitz. Since the region bounded by the graph of $L$-Lipschitz function and $x$-axis is $(L+1)$-quasiconvex, the mapping $\Phi_0$ is $\left((L+1)\cdot L'(L,C)\right)$-biLipschitz.

Let $\phi = \Phi_0\big|_{[0;1]\times\{0\}}.$ By construction $\phi$  is piecewise linear self-homeomorphism of the segment $[0;1]$, and $C^{-2}\le\phi'(x)\le C^2$ for almost all $x\in[0;1].$ We note that $\Phi_0 (x, 0) = (\phi(x), 0)$ and $\Phi_0(x, S) = (\phi(x), f(\phi(x))).$ Therefore to prove the lemma it remains to construct an area-preserving piecewise affine self-homeomorphism $\widetilde\Phi$ of the rectangle $[0;1]\times[0;S]$, whose biLipschitz constant is not greater than some a priori given function on $L$ and $C$, and whose restriction on vertical sides of the rectangle is identity, and the restriction on horizontal sides coincides with $\phi.$ Then $\Phi_0\circ\widetilde\Phi^{-1}$ will be the required mapping.

\smallskip
\noindent{\bf Step 2.} We decompose $\phi = \Phi_0\big|_{[0;1]\times\{0\}}$ into a composition of $3/2$-biLipschitz piecewise linear orientation preserving self-homeomorphisms of the segment $[0;1],$ whose number is bounded above by a constant depending only on $C.$

Suppose that the biLipschitz constant of $\phi$ is greater than 3/2.
Let
$$X = \{x\in[0;1]:\ \phi'(x) > 3/2\},\ Y = \{x\in[0;1]:\ \phi'(x) < 1\}.$$ 

Since $\int_0^1 \phi'(x)dx = 1,$ then $\ell(X)\le 2\ell(Y),$ where $\ell(A)$ is the measure of the set $A.$ Let $Z\subset Y$ be a subset of measure $\ell(X)/2$ which is a finite union of segments. Let $\phi_1$ be a homeomorphism of the segment $[0;1]$ such that
$$\phi_1'(x) =
\left\{
\begin{aligned}
7/6,&\ x\in X,\\
2/3,&\ x\in Z,\\
1,&\ \text{in other cases}.
\end{aligned}
\right.
$$

We see that $\phi_1$ is piecewise linear and $3/2$-biLipschitz, $\mathrm{lip}\left(\phi\circ\phi_1^{-1}\right) \le \max(6/7\cdot\mathrm{lip}(\phi), 3/2)$ and $\mathrm{lip}(\phi^{-1})\ge\mathrm{lip}\left(\phi_1\circ\phi^{-1}\right).$ We can repeat this procedure for the mapping $\phi\circ\phi_1^{-1}$ and so on. We will get
$$
\phi = \widetilde\phi\circ\phi_m\dots\circ\phi_2\circ\phi_1,
$$
where $m = \lceil\log_{7/6}\left(\mathrm{lip}(\phi)\cdot 2/3\right)\rceil,$ $\phi_k$ are piecewise linear $3/2$-biLipschitz increasing homeomorphisms, $\widetilde \phi$ is piecewise linear $3/2$-Lipschitz homeomorphism, such that $\mathrm{lip}\left(\widetilde\phi^{-1}\right)\le\mathrm{lip}(\phi^{-1}).$ We denoted by $\lceil x\rceil$ the least possible integer $n$ such that $n\ge x.$

Then we can repeat these arguments for the mapping $\widetilde \phi^{-1}$ and we will get a sought decomposition. On this step we showed that we can assume that $\phi$ is $3/2$-biLipschitz.

\smallskip
\noindent{\bf Step 3.} On this step we reduce the problem to the case $S=2.$ We apply an affine transformation $A,$ which maps the rectangle $[0;1]\times[0;S]$ to the rectangle $[0;1]\times[0;2].$ If $\widetilde\Phi$ is a sought mapping for the rectangle $[0;1]\times[0;2],$ then the mapping $A^{-1}\widetilde\Phi A$ is a self-homeomorphism of $[0;1]\times[0;S]$, it preserves area since $\det A$ is constant, its restriction to horizontal (respectively, vertical) sides is $\phi$ (respectively, affine), and $\mathrm{bilip}\left(A^{-1}\widetilde\Phi A\right)\le \mathrm{bilip}(A)^2\cdot\mathrm{bilip}\left(\widetilde\Phi\right)\le (2C)^2\cdot\mathrm{bilip}\left(\widetilde\Phi\right).$

\smallskip
\noindent{\bf Step 4.}
We construct a self-homeomorphism $\widetilde\Phi$ of the rectangle $[0;1]\times[0;2]$ such that $\widetilde\Phi(x,0) = (\phi(x),0)$ and $\widetilde\Phi(x,2) = (\phi(x),2)$ if $x\in[0;1]$, $\widetilde\Phi(0,y) = (0,y)$ and $\widetilde\Phi(1,y) = (1, y)$ if $y\in[0;2],$ where $\phi$ is a piecewise linear orientation preserving $3/2$-biLipschitz self-homeomorphism of $[0;1].$

First we extend the homeomorphism $\phi$ to the whole line. We define $\phi(x) = \phi(-x) + 2x$ on the segment $[-1;0].$ We have $\phi(-1) = -1,$ $\phi(0) = 0$ and $\phi'(x) = -\phi'(-x) + 2 \ge 1/2$ so $\phi$ is strictly monotonic and continuous. We define $\phi$ on the segment $[2k - 1; 2k + 1],$ where $k$ is nonzero integer, by $\phi(x) = \phi(x - 2k) + 2k.$ We note that for any $x\in\mathbb R$
$$
\phi(x) = \phi(-x) + 2x.
$$

Indeed, on the segment $[-1;1]$ this is true by definition. If $x = x_0 + 2k,$ where $x_0\in[-1;1]$ and $k$ is an integer, then
$$\phi(x) - \phi(-x) = \phi(x_0 + 2k) - \phi(-x_0 -2k) = \phi(x_0) + 2k - \phi(-x_0) + 2k = 2x_0 + 4k = 2x.$$

Now we apply Proposition~\ref{from-line-to-plane}. Let $\Psi:(x,y)\mapsto(\widetilde x,\widetilde y)$ be a mapping defined by~(\ref{line-to-line}). Then
$$
x(\widetilde x, 2) = \frac{\phi(\widetilde x + 2) + \phi(\widetilde x - 2)}2 = \frac{\phi(\widetilde x) + 2 + \phi(\widetilde x) - 2}2 = \phi(\widetilde x) = x(\widetilde x, 0)
$$
and
$$
\widetilde y(\widetilde x, 2) = \frac{\phi(\widetilde x + 2) - \phi(\widetilde x - 2)}2 = \frac{\left(\phi(\widetilde x) + 2\right) - \left(\phi(\widetilde x) - 2\right)}2 = 2.
$$

So $\Psi(x,0) = (\phi^{-1}(x),0)$ and $\Psi(x,2) = (\phi^{-1}(x), 2)$ for any $x\in[0;1].$

We also have
$$
\widetilde y(\widetilde x = 0, y) = \frac{\phi(y) - \phi(-y)}2 = y
$$
and
$$
\widetilde y(\widetilde x = 1, y) = \frac{\phi(1+y) - \phi(1-y)}2 = \frac{2 + \phi(-1+y) - \phi(1-y)}2 = \frac{2 + 2(-1+y)}2 = y.
$$

\begin{figure}
\center
\includegraphics{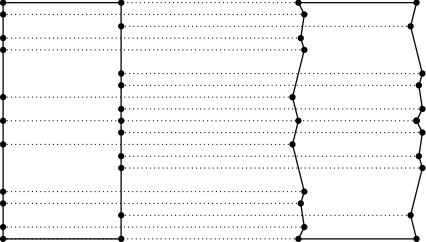}
\caption{The rectangle $[0;1]\times[0;2]$ is on the left, its image under $\Psi^{-1}$ is on the right.}
\label{psi-of-rectangle}
\end{figure}

However $\Psi$ does not map the rectangle $[0;1]\times[0;2]$ to itself in general, since it is possible that $x(0, y)\neq 0$ (Figure~\ref{psi-of-rectangle}). Indeed
$$
x(0,y) = \frac{\phi(y)+\phi(-y)}2 = \phi(y) - y
$$
and
\begin{multline*}
x(1,y) = \frac{\phi(1+y)+\phi(1-y)}2 = \frac{2+\phi(-1+y)+\phi(1-y)}2 = \\
= \frac{2+\phi(-1+y)+ \phi(-1+y) - 2(-1+y)}2 =
\phi(-1+y) + 2 - y = \phi(1+y) - y.
\end{multline*}

The boundary of the polygon $\Psi^{-1}([0;1]\times[0;2])$ consists of two horizontal segments and four more parts isometric to each other. In a moment we will define skew-mappings $\overline \Phi_0,$ $\overline \Phi_1,$ $\overline \Phi_2$, $\overline \Phi_3$ from Lemma~\ref{skew-lemma} such that the composition $\overline\Phi_2^{-1}\circ\overline\Phi_3\circ\overline\Phi_2\circ\overline\Phi_0^{-1}\circ\overline\Phi_1\circ\overline\Phi_0$ maps this polygon to the rectangle such that any point on the boundary preserves its $y$-coordinate and every point on the horizontal part of the boundary is fixed. Let
$$
\overline \Phi_0:
\quad
x \mapsto x,
\quad
y \mapsto
\left\{
\begin{aligned}
y,&\text{ if } x\le 1/4,\\
-2x + y + 1/2,&\text{ if } 1/4 \le x \le 3/4,\\
y - 1,&\text{ if } 3/4 \le x
\end{aligned}
\right.
$$
and
$$
\overline\Phi_1:\quad
x\mapsto
\left\{
\begin{aligned}
x,&\text{ if } y\le 0 \text{ or }1\le y,\\
x + y - \phi(y),&\text{ if } 0\le y\le 1,
\end{aligned}
\right.
\quad
y \mapsto y.
$$

\begin{figure}[h]
\includegraphics{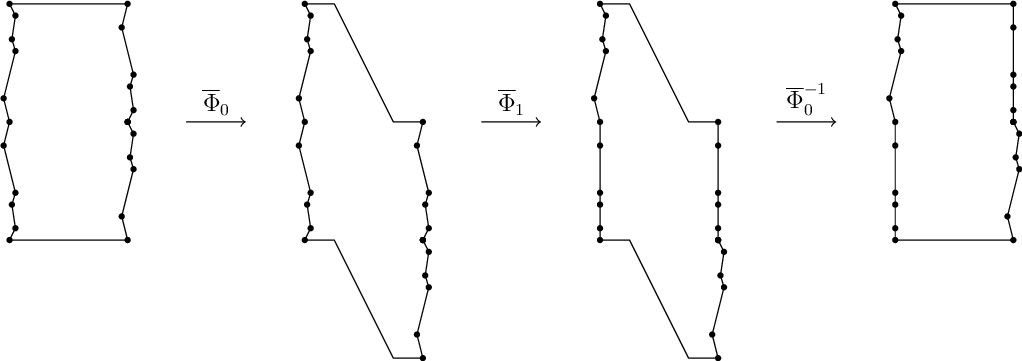}
\caption{Straightening of the polygon $\Psi^{-1}([0;1]\times[0;2])$.}
\label{three-skews}
\end{figure}

We apply $\overline\Phi_0^{-1}\circ\overline\Phi_1\circ\overline\Phi_0$ to the polygon $\Psi^{-1}([0;1]\times[0;2])$ (Figure~\ref{three-skews}). If a horizontal segment on the boundary of the polygon lies inside a region $y\ge 1$ (respectively, $y\le 0$) then its image under $\overline\Phi_0$ lies in the same region. Therefore $\overline\Phi_1$ acts trivially on the images of horizontal segments and the mapping $\overline\Phi_0^{-1}\circ\overline\Phi_1\circ\overline\Phi_0$ acts trivially on horizontal segments. Now let us see how the mapping acts on the left part of the boundary. Since $\phi'(y) - 1 \le 1/2$ for all $y$, $\phi(0) = 0$, $\phi(1) = 1$ and $\phi(2)=2,$ then $\phi(y) - y\le 1/4$ if $0\le y\le 2.$ Therefore $\overline\Phi_0$ acts trivially on the left part of the boundary. The upper half of the left part of the boundary stays fixed under $\overline\Phi_1$. The lower half of the left part of the boundary is mapped to a lower half of the left side of the rectangle while $y$-coordinate of any point does not change. After this $\overline\Phi_0^{-1}$ does nothing. Now let us see how the mapping acts on the right part of the boundary. Since $\phi(1+y) - y=1$ if $y=0,$ $y=1$ or $y=2$, and since $\phi' \ge 1/2$ for all $y$, we have $\phi(1+y)-y\ge3/4$ if $y\in[0;2].$ Therefore the mapping $\overline\Phi_0$ acts on the right part as a parallel translation by a vector $(0,-1).$ So if $y\in[1;2]$
\begin{multline*}
(\phi(1+y) - y, y) \xmapsto{\overline\Phi_0} (\phi(1+y) - y, y - 1) \xmapsto{\overline\Phi_1} (\phi(1+y) - y + \left((y-1) - \phi(y-1)\right), y-1) = \\
= ( \left(2 + \phi(-1 + y)\right) - 1 - \phi(y-1), y-1 ) = (1,y-1) \xmapsto{\overline\Phi_0^{-1}}
(1,y).
\end{multline*}

So the upper half of the right part is mapped to the line $x = 1$ preserving the $y$-coordinate. The lower half stays fixed under $\overline\Phi_0^{-1}\circ\overline\Phi_1\circ\overline\Phi_0$. Similarly we can choose skew-mappings $\overline \Phi_2$ and $\overline \Phi_3$ to straighten the remaining part of the boundary of the polygon so that the mapping $\overline\Phi_2^{-1}\circ\overline\Phi_3\circ\overline\Phi_2\circ\overline\Phi_0^{-1}\circ\overline\Phi_1\circ\overline\Phi_0\circ\Psi^{-1}$ is the required mapping.
\end{proof}

We denote a homothety $z\mapsto \lambda z$ by $\alpha_{\lambda}.$ The following lemma is a reformulation of the previous one.

\begin{lemm}
\label{square-to-graph}
Let $f$ be a piecewise linear $L$-Lipschitz function on a segment $[a; b]$. Let $h$, $C_1$ and $C_2$ be positive constants such that $h/C_2\le b-a \le h\cdot C_2$ and $h/C_1 \le f(x)\le h\cdot C_1$ for every $x\in[a;b].$ Let $S = \int_a^b f(x) dx.$ Then there exists a $K$-biLipschitz piecewise affine mapping $\Phi: [a;b]\times [0;h] \to \mathbb R^2,$ where $K$ depends only on $L,$ $C_1$ and $C_2$, such that
\begin{enumerate}[label = \arabic*),ref = \arabic*)]
\item
the composition $\Phi\circ\alpha_{\sqrt{\frac{S}{(b-a)h}}}$ preserves area;
\item
$\Phi(x, 0) = (x, 0)$ and $\Phi(x, h) = (x, f(x))$ for any $x\in[a;b]$;
\item
$\Phi(a, y) = \left(a, y\cdot\frac{f(a)}{h}\right)$ and $\Phi(b, y) = \left(b, y\cdot\frac{f(b)}{h}\right)$ for any $y\in[0;h].$
\end{enumerate}
\end{lemm}

\begin{lemm}
\label{square-to-graph2}
Let $f_1$ and $f_2$ be piecewise linear $L$-Lipschitz functions on a segment $[a; b]$. Let $h$, $C_1$ and $C_2$ be positive constants such that $h/C_2\le b-a \le h\cdot C_2$ and $h/C_1 \le f_1(x) - f_2(x)\le h\cdot C_1$ for every $x\in[a;b].$ Let $A = \int_a^b f_1(x) - f_2(x) dx.$ Then there exists a piecewise affine $K$-biLipschitz mapping $\Phi: [a;b]\times [0;h] \to \mathbb R^2,$ where $K$ depends only on $L,$ $C_1$ and $C_2$, such that
\begin{enumerate}[label = \arabic*),ref = \arabic*)]
\item
a composition $\Phi\circ\alpha_{\sqrt{\frac{A}{(b-a)h}}}$ preserves area;
\item
$\Phi(x, 0) = (x, f_2(x))$ and $\Phi(x, h) = (x, f_1(x))$ for any $x\in[a;b]$;
\item
$\Phi(a, y) = \left(a, \frac{y}{h}\cdot f_1(a) + \left(1-\frac{y}{h}\right)\cdot f_2(a)\right)$ and $\Phi(b, y) = \left(b, \frac{y}{h}\cdot f_1(b) + \left(1-\frac{y}{h}\right)\cdot f_2(b)\right)$ for any $y\in[0;h].$
\end{enumerate}
\end{lemm}

\begin{proof}
We apply Lemma~\ref{square-to-graph} to a function $f_1 - f_2$ and we obtain a mapping $\Phi_0$ which satisfies all conditions except that its image is a region bounded by the graph of $f_1 - f_2$ and $x$-axis, but we need a region bounded by graphs of the functions $f_1$ and $f_2.$ We map the first region to the second one by
$$
\Phi_1: (x,y) \mapsto (x, y + f_2(x)).
$$

By Lemma~\ref{skew-lemma} this mapping preserves area, is piecewise affine and $\sqrt{2+L^2}$-biLipschitz so the mapping $\Phi = \Phi_1\circ\Phi_0$ has desired properties.

\end{proof}

We will use the following lemma only in Section~\ref{correction-on-some-square}.

\begin{lemm}
\label{square-rotation}
Let $B$ be a square on the plane, $U$ be its open neighborhood. There exists $\varepsilon > 0$ such that for any $\varphi\in (-\varepsilon;\varepsilon)$ there exists a piecewise affine area-preserving 2-biLipschitz mapping $\Phi:\mathbb R^2\to\mathbb R^2$ which is the identity outside $U$ and whose restriction to $B$ coincides with a rotation by $\varphi$ around its center.
\end{lemm}

\begin{figure}[h]
\center
\includegraphics{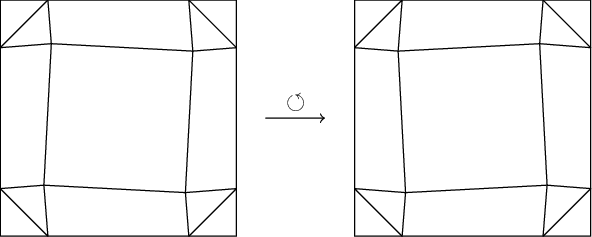}
\caption{A piecewise affine mapping whose restriction to the central square is rotation.}
\label{pl-rotation}
\end{figure}

\begin{proof}
We fix some $\varphi<90^{\circ}.$
Let $P$ be some vertex of the square $B$, $P'$ its image under a rotation of $P$ by $\varphi$ about the center $O$ of $B.$ Let $B'$ be a square with the center $O$ containing $B$ and whose diagonal is parallel to the line $PP'.$ If $\varphi$ is sufficiently small, we can also assume that $B'$ lies inside $U$. On two nearest sides of the square $B'$ to the point $P$ there exists the unique pair of points $P_1,$ $P_2$ such that the segment $P_1 P_2$ is parallel to the line $PP'$ and $\angle P_1 P P_2 = 90^{\circ}.$  Then the square $B'$ is tiled into the square $B$, four equal triangles $PP_1P_2$ and its images under rotations by $90^{\circ},$ $180^{\circ}$ and $270^{\circ},$ four more isosceles right triangles at the vertices of $B'$ and four trapezoids (Figure~\ref{pl-rotation}, on the left).

We construct one more tiling of $B'$ (Figure~\ref{pl-rotation}, on the right). It consists of the image of $B$ under rotation by $\varphi$ and similar triangles and trapezoids. The new tiling can be obtained from the previous one by means of symmetry. We consider another correspondence between elements of these tilings: the one which is the rotation on the square $B$ and which is the identity on the boundary of $B'.$ Due to symmetry, the areas of the corresponding elements are equal and the heights of trapezoids are equal. We consider the mapping  from $B'$ to itself which induces the mentioned correspondence, which is affine on $B$ and on each triangle, and which is defined on trapezoids by applying Lemma~\ref{trapezoid-lemma} two times. Despite that these trapezoids do not satisfy conditions of the lemma (there exists an obtuse angle on the longer base and the heights are too big with respect to the bases) we can repeat the construction in the lemma but we can not use the estimates of that lemma. Since the constructed mapping is close to the identity if $\varphi$ is small, and its Jacobi matrix is continuous on $\varphi$ in the interior of each element of the tiling, for sufficiently small $\varepsilon>0$ the mapping is 2-biLipschitz for all $\varphi\in(-\varepsilon;\varepsilon).$
\end{proof}

%


\section{Proof of the main theorem}
\label{main-section}
In this section we prove Theorems~\ref{main-theorem},~\ref{annulus-version} and~\ref{the-technical-lemma-without-parameters}.

\subsection{A result of Tukia and an extension to the square}
\label{square-extension}
This section is devoted to the proofs of Propositions~\ref{jacobian-problem} and~\ref{square-extension-lemma}.

\begin{prop}
\label{square-extension-lemma}
For any $L$ there exists $\widetilde K$ such that any $L$-biLipschitz mapping $\phi:\partial B\to\mathbb R^2$ such that the curve $\phi(\partial B)$ bounds a region of area 4 can be extended to an area-preserving  $\widetilde K$-biLipschitz mapping of the square $B$ which is locally piecewise affine on the interior of the square.
\end{prop}

\begin{figure}
\center
\includegraphics{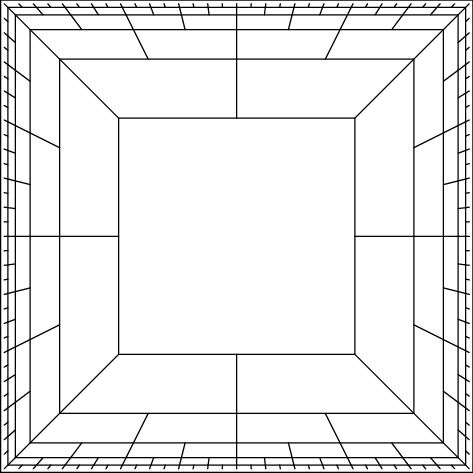}
\hspace{1cm}
\includegraphics{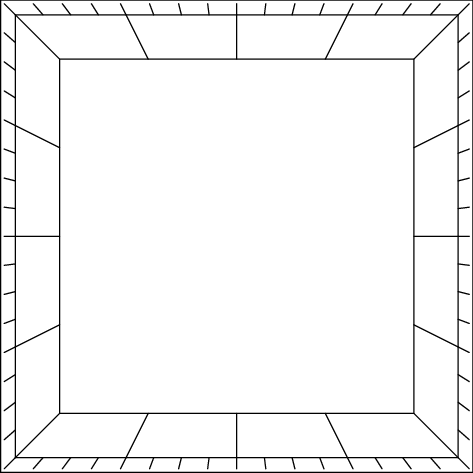}
\caption{The tiling of the square into quadrilaterals $Q_j$ for $m=1$ and for $m=2.$}
\label{tiling}
\end{figure}

Let $m\ge1$ be an integer. Let $B = [-1;1]^2$ be a square on the plane, $B_i = \left(1-2^{-im}\right)\cdot B, i\ge1$ be a sequence of squares obtained from $B$ by homotheties. We divide $B$ into the square $B_1$ and annuli $B_{i+1}\setminus B_i.$ The boundary of the annulus $B_{i+1}\setminus B_i$ is a union of the boundaries of two squares. We divide each side of these squares into $2^{im}$ equal segments as in Figure~\ref{tiling}, and we divide the annulus respectively into $4\cdot2^{im}$ quadrilaterals by radial line segments. We denote the square $B_1$ by $Q_0$ and we arbitrarily number quadrilaterals of all annuli by positive integers and denote them by $Q_j.$

Let $\lambda_{j} = \mathrm{diam}(Q_j),$ where $\mathrm{diam}(X)$ denotes the diameter of the set $X.$ Take another copy of the square $B$ and divide it into two triangles by a diagonal. For each quadrilateral $Q_{j}$ we fix a homeomorphism $q_{j}:\lambda_{j}\cdot B\to Q_{j}$ which is affine on each of two triangles.

It is clear that all the mappings $q_{j}$ are $M$-biLipschitz for some constant $M$, which is independent of $m.$

Let $\alpha_{\lambda}:z\mapsto \lambda z$ be the homothety with the coefficient $\lambda$ and with the center in the origin of the coordinate system. By a {\it similarity} we call a composition of an isometry and a homothety.

\begin{lemm}
\label{Tukia-theorem}
For any $L$ there exist an integer $m$, a real number $K$ and a finite number of piecewise affine embeddings $p_k:B \to\mathbb R^2$ such that any $L$-biLipschitz mapping $\phi: \partial B \to \mathbb R^2$ can be extended to such $K$-biLipschitz mapping $\Phi:B\to\mathbb R^2$ that for any quadrilateral $Q_{j}$ (from the construction above which depends on $m$) there exist such $k$ and a similarity $s_{j}$ that
$$
\Phi\big|_{Q_{j}}\circ q_{j} = s_{j}\circ p_k \circ\alpha_{\lambda_{j}^{-1}}.
$$
\end{lemm}

Lemma~\ref{Tukia-theorem} is proved by P.\,Tukia. It is a compilation of Lemmas~2A and~3A and results of Section~12 in~\cite{Tuk}. 

\begin{rema}
It can be shown that actually in Lemma~\ref{Tukia-theorem} one can always assume that $m=1$. We will not use this fact in this paper.
\end{rema}

By Lemma~\ref{Tukia-theorem} the bounded component of $\mathbb R^2\setminus\phi(\partial B)$ is tiled into polygons $\Phi(Q_j)$ which form a finite set of similarity classes, i. e. equivalence classes with respect to similarities. In the proof of Theorem~\ref{main-theorem} we use this tiling and we will not change it. To make $\Phi$ area-preserving we change the tiling of $B$. We will construct new tiling into polygons $\widetilde Q_j$ which have the same area as $\Phi(Q_j).$ The image of $\widetilde Q_j$ under a new mapping will be the same as in Lemma~\ref{Tukia-theorem}, i. e. $\Phi(Q_j).$



\begin{lemm}
\label{half-theorem}
For any $\widetilde C>1$ there exists a constant $\widetilde M$ such that for any positive real numbers $A_j, j\ge 0,$ such that
\begin{itemize}
\item $1/\widetilde C\le \frac{A_j}{\mathrm{Area}(Q_j)}\le \widetilde C$ for any $j\ge0;$
\item $\sum\limits_{j=0}^{\infty}A_j = 4$
\end{itemize}
there exist polygons $\widetilde Q_j, j\ge0$ and area-preserving piecewise affine $\widetilde M$-biLipschitz mappings $\widetilde q_{j}:\widetilde\lambda_{j}\cdot B \to\widetilde Q_j, j\ge 0$, where $\widetilde\lambda_{j} = \frac12\sqrt{A_j}$, such that
\begin{enumerate}[label = \arabic*), ref = \arabic*)]
\item
\label{combinatorial-equivalence}
for all $j$ and $l$
$$\left(\widetilde q_{j}\circ\alpha_{\widetilde\lambda_{j}}\right)^{-1}\circ\left(\widetilde q_{l}\circ\alpha_{\widetilde\lambda_{l}}\right)
=
\left(q_{j}\circ\alpha_{\lambda_{j}}\right)^{-1}\circ\left(q_{l}\circ\alpha_{\lambda_{l}}\right),$$
where each side of the equation is considered as a relation over $B$ (in fact it is a bijection between possibly empty subsets of $B$);
\item
\label{interior-cover}
the interior of $B$ is $\bigcup\limits_j \widetilde Q_j;$
\item
\label{boundary-limit}
a sequence $z_n \in Q_{j_n}$ converges to a point $z\in \partial B$ if and only if any sequence $\widetilde z_n \in \widetilde Q_{j_n}$ converges to $z;$
\item
$\widetilde q_0$ is affine and $\widetilde q_0\left(\widetilde\lambda_0\cdot B\right)$ is a rectangle whose sides are parallel to the sides of $B$.
\end{enumerate}
\end{lemm}

Condition~\ref{combinatorial-equivalence} means that the tiling of $B$ into polygons $\widetilde Q_{j}$  is combinatorially equivalent to the tiling into quadrilaterals $Q_{j}.$ In particular, polygons $\widetilde Q_j$ and $\widetilde Q_l$ do not intersect if polygons $Q_j$ and $Q_l$ do not intersect. If quadrilaterals $Q_j$ and $Q_l$ have a common side, polygons $\widetilde Q_j$ and $\widetilde Q_l$ have a common arc of their boundaries, each preimage of this arc under the mappings $\widetilde q_{j}$ and $\widetilde q_{l}$ is a side of the square and the composition $\widetilde q_{j}^{-1}\circ \widetilde q_{l}$ is a linear mapping from one side to the second.

\begin{proof}[Proof of Lemma~\ref{half-theorem}] 
In Figure~\ref{tiling} the square $B$ is divided into the central square $Q_0$ and four trapezoids $T_{n}$, $n=1,2,3,4$, one base of which coincides with a side of the square $Q_0$, and the second one coincides with a side of the square $B$.

\begin{figure}[h]
\center
\includegraphics{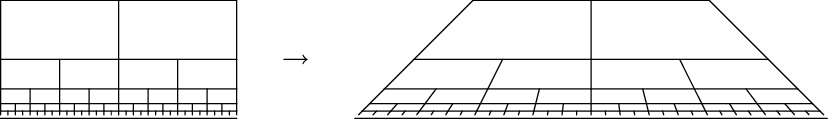}
\caption{The mapping $t_n:R\to T_n$ for $m=1$. The trapezoid is tiled into quadrilaterals $Q_j$.}
\label{rectifying-trapezoid}
\end{figure}

Let $R = [0;1]\times[0;2^{-m}]\subset\mathbb R^2.$ 
We fix such mappings $t_{n}: R\to T_{n}$ that their restriction on each side of the rectangle $R$ and on each horizontal line is affine and the lower side (of length 1) of the rectangle $R$ is mapped to the longer base of the trapezoid (Figure~\ref{rectifying-trapezoid}).

\begin{figure}[h]
\center
\includegraphics{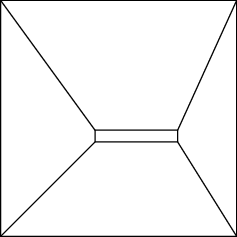}
\caption{The square is divided into the rectangle $\widetilde Q_0$ and the trapezoids $\widetilde T_1$, $\widetilde T_2$, $\widetilde T_3$, $\widetilde T_4$.}
\label{tiling-area}
\end{figure}

\smallskip
\noindent{\bf Step 1.}
We divide the square $B$ into a rectangle $\widetilde Q_{0}$ and four trapezoids $\widetilde T_{n}$ such that
$$
\mathrm{Area}\left(\widetilde Q_0\right) = A_0,\ \mathrm{Area}\left(\widetilde T_{n}\right) = \sum\limits_{Q_j\subset T_n}A_j, n = 1,2,3,4,
$$
one base of each trapezoid coincides with a side of the square $B$, the second one coincides with a side of the rectangle $\widetilde Q_0$, and trapezoids $\widetilde T_{n}$ are numbered in the same way as the trapezoids $T_{n}.$ This can be done in a unique way (Figure~\ref{tiling-area}).

By assumptions of the lemma we have $A_j\ge \mathrm{Area}(Q_j)/\widetilde C$, so we have $\mathrm{Area}\left(\widetilde Q_0\right)\ge\mathrm{Area}(Q_0)/\widetilde C$ and $\mathrm{Area}\left(\widetilde T_n\right)\ge \mathrm{Area}(T_n)/\widetilde C$ for any $n=1,2,3,4.$ Therefore the lengths of sides of the rectangle and of heights of constructed trapezoids are separated from zero by a real number that depends only on $\widetilde C.$ So we can choose area-preserving piecewise affine mappings
$$
\widetilde q_{0}:\widetilde\lambda_0\cdot B\to \widetilde Q_{0},\ \widetilde t_{n}:\sqrt{S_{n}\cdot2^{m}}\cdot R\to \widetilde T_{n},
$$
where $S_{n} = \mathrm{Area}\left(\widetilde T_{n}\right)$, in such a way that the biLipschitz constants of these mappings are bounded above by some function of $\widetilde C$ (since the trapezoids have acute angles on the longer base, we can apply Lemma~\ref{trapezoid-lemma}).

We can also choose a mapping $\widetilde q_0$ to be affine and so that for each side of the square $B$ its image under the mapping $\widetilde q_0\circ\alpha_{\widetilde \lambda_0}$ lies to the same side from the rectangle $\widetilde Q_0$ as its image under $q_0\circ\alpha_{\lambda_0}$ lies from the square $Q_0$: respectively above, below, to the left or to the right. We also can choose the mappings $\widetilde t_{n}$ in such a way that the lower side is mapped to the longer base of the trapezoid, mappings $\widetilde t_{n}\circ\alpha_{\sqrt{S_{n}\cdot2^{m}}}$ are affine on the sides of the rectangle $R$ and
$$
\left(\widetilde t_{n}\circ\alpha_{\sqrt{S_{n}\cdot2^{m}}}\right)^{-1}\circ \left(\widetilde t_{k}\circ\alpha_{\sqrt{S_{k}\cdot2^{m}}}\right) = t_{n}^{-1}\circ t_{k}
$$
for any $n,k\in\lbrace1,2,3,4\rbrace.$

\smallskip
\noindent{\bf Step 2.} Before we tile the trapezoids $\widetilde T_n$ into polygons $\widetilde Q_j,$ we  construct auxiliary polygons $\overline Q_j$ which form a tiling (of another region) combinatorially equivalent to the tiling of the trapezoid into quadrilaterals $Q_j$ and such that $\mathrm{Area}(\overline Q_j) = A_j.$ Their union will be a region bounded by a graph of some function on the segment $[0;1]$ and the $x$-axis.

We fix $n\in\lbrace1,2,3,4\rbrace$.

For each quadrilateral $Q_j\subset T_n$ we construct a polygon $\overline Q_j$ which has two common vertical sides with the rectangle $t_n^{-1}(Q_j)$ (Figure~\ref{rectifying-trapezoid}). Endpoints of these sides have the form

$$
((k-1)\cdot 2^{-im}, 2^{-im}),\ \left((k-1)\cdot 2^{-im}, 2^{-(i+1)m}\right),\ (k\cdot 2^{-im}, 2^{-im}),\ \left(k\cdot 2^{-im}, 2^{-(i+1)m}\right)
$$
for some $i$ and $k.$
The rest part of the boundary of $\overline Q_j$ is a union of two graphs of piecewise linear functions.

To construct these functions we apply Lemma~\ref{main-lemma}. First we need to define real numbers $A_{N,k}$ from assumptions of the lemma.

A pair of integers $(N, k)$ is called {\it admissible} if $N\ge 0$ and $k\in\left\lbrace1, 2, \dots, 2^N\right\rbrace.$

\begin{figure}
\center
\includegraphics{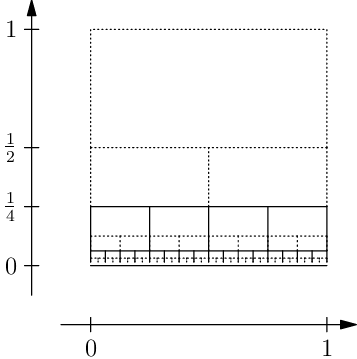}
\put(-91,125){$R(0,1)$}
\put(-119,83){$R(1,1)$}
\put(-62,83){$R(1,2)$}
\caption{The boundary of rectangles $t_n^{-1}(Q_j)$ is depicted by solid lines. Solid and dashed lines bound rectangles $R(N,k)$. The case $m=2$.}
\label{R-n-k}
\end{figure}

For each admissible pair $(N, k)$ we define a rectangle $R(N,k)$ with vertices
$$
\left((k-1)\cdot 2^{-N}, 2^{-N}\right),\ \left((k-1)\cdot 2^{-N}, 2^{-N-1}\right),\ \left(k\cdot 2^{-N}, 2^{-N}\right),\ \left(k\cdot 2^{-N}, 2^{-N-1}\right),
$$
see~Figure~\ref{R-n-k}.

If $N<m$ we set $A_{N,k} = 2^{-2N}/2.$

If $N = i\cdot m + r$, where $i,r$ are integers, $i\ge1,$ $0\le r<m$  and $R(N,k)\subset t_n^{-1}(Q_j),$ we set
$$
A_{N,k} = A_j\cdot\frac{\mathrm{Area}(R(N,k))}{\mathrm{Area}\left(t_n^{-1}(Q_j)\right)} = A_j\cdot2^{-2r-1} /\left(1 - 2^{-m}\right).
$$

Note that
\begin{equation}
\label{area-equation}
\sum\limits_{R(N,k)\subset t_n^{-1}\left(Q_j\right)} A_{N,k} = A_j.
\end{equation}

Now we check the assumptions of Lemma~\ref{main-lemma}.

Let quadrilateral $Q_j$ lie in an annulus $B_{i+1}\setminus B_i.$
By construction
$$
\mathrm{Area}(Q_j) = \left(\left(1 - 2^{-(i+1)m}\right)^2 - \left(1 - 2^{-im}\right)^2\right)\cdot 2^{-im}.	
$$

Then
\begin{multline*}
\mathrm{Area}(Q_j) = \left( 2\cdot \left( 2^{-im} - 2^{-(i+1)m} \right) + 2^{-2(i+1)m} - 2^{-2im} \right)\cdot 2^{-im} = \\
 = 2^{-2im}\cdot\left(2\cdot(1-2^{-m}) - 2^{-im}\cdot(1 - 2^{-2m}) \right).
\end{multline*}

It is clear that $\mathrm{Area}(Q_j) < 2\cdot2^{-2im}.$
Also
$$
\mathrm{Area}(Q_j)\ge 
2^{-2im}\cdot\left(2\cdot(1-2^{-m}) - 2^{-m}\cdot(1 - 2^{-2m}) \right) = 
2^{-2im}\cdot\left(2 - 3\cdot2^{-m} + 2^{-3m} \right) > 2^{-2im}/2.
$$

So
$$
\left(2\widetilde C\right)^{-1}\cdot 2^{-2im} < A_j < 2\widetilde C\cdot 2^{-2im}.
$$

Therefore
$$
A_{N,k} < 2\widetilde C\cdot2^{-2im}\cdot2^{-2r-1} /\left(1 - 2^{-m}\right) = \widetilde C\cdot2^{-2N}/(1-2^{-m}) 
\le 2\widetilde C\cdot2^{-2N}$$
and
$$
A_{N,k} > \left(2\widetilde C\right)^{-1}\cdot 2^{-2im}\cdot2^{-2r-1} /\left(1 - 2^{-m}\right) = \left(4\widetilde C\right)^{-1}\cdot2^{-2N}/\left(1 - 2^{-m}\right) > \left(4\widetilde C\right)^{-1}\cdot2^{-2N}.
$$

So we can use Lemma~\ref{main-lemma} for $C = 4\widetilde C.$ Let $f_{N,k}$ be functions given by the lemma.

\begin{figure}[h]
\center
\includegraphics{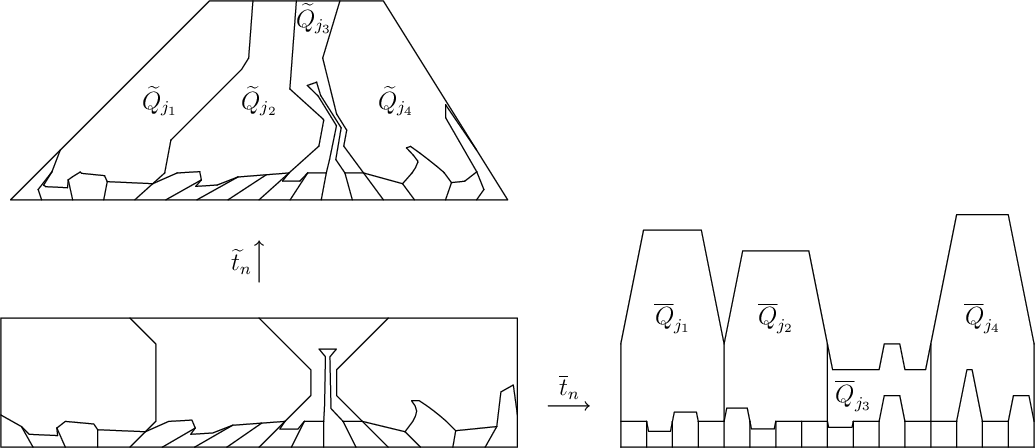}
\caption{Mappings $\overline t_n: \sqrt{S_n\cdot 2^{m}}\cdot R\to \overline T_n$ and $\widetilde t_n: \sqrt{S_n\cdot 2^{m}}\cdot R\to \widetilde T_n.$ We define a polygon $\widetilde Q_j$ as the image of the polygon $\overline Q_j$ under the mapping $\widetilde t_n\circ\overline t_n^{-1}.$ On the right $2^{m}$ top polygons $\overline Q_j$ are shown. The case $m=2.$}
\label{t-overline-tilde}
\end{figure}

Now we define polygons $\overline Q_j$. We already defined the vertical part of their boundaries. Let $\left(k\cdot2^{-im}, 2^{-im}\right)$ be the right top corner of $\overline Q_j.$ We define the upper part of the boundary as a graph of the function $f_{im, k}.$ We define the lower part as a union of graphs of functions $f_{i(m+1), l},$ where $(k-1)\cdot2^m \le l \le k\cdot2^{m}$ (Figure~\ref{t-overline-tilde}, on the right). This union is a continuous curve by condition~\ref{mlp3} of Lemma~\ref{main-lemma}.

The area of a polygon $\overline Q_j$ equals $A_j$ by equality~(\ref{area-equation}) and property~\ref{mlp5} of Lemma~\ref{main-lemma} (to see this sum up the equality in \ref{mlp5} for all admissible pairs $(N,k)$ such that $R(N,k)\subset t_n^{-1}(Q_j)$).

\smallskip
\noindent{\bf Step 3.} On this step we define polygons $\widetilde Q_j.$

Let $n\in\lbrace1,2,3,4\rbrace$ be the number that was fixed on the previous step.
Let $\overline T_n$ be the union of polygons $\overline Q_j.$ The polygon $\overline T_n$ is a union of regions bounded by graphs of functions $f_{m, l},$ where $1\le l \le 2^m,$ and $x$-axis. 

Functions $f_{m,l}$ are piecewise linear and $8C$-Lipschitz by property~\ref{mlp2} of Lemma~\ref{main-lemma}, are at least $2^{-m}/8C$ at each point by property~\ref{mlp4} and equal $2^{-m}$ at the endpoints by property~\ref{mlp3}. So by Lemma~\ref{square-to-graph} there exists a mapping $\overline t_n: \sqrt{S_n\cdot 2^{m}}\cdot R\to \overline T_n$ (Figure~\ref{t-overline-tilde}) such that
\begin{enumerate}[label = \arabic*), ref = \arabic*)]
\item the left (respectively, the right) vertical side of the rectangle is mapped to the set $\{x = 0\}\cap \overline T_n$ (respectively, $\{x = 1\}\cap \overline T_n$);
\item the lower side of the rectangle is mapped to the set $\{y = 0\}\cap \overline T_n;$
\item the restriction of $\overline t_n$ on the left, on the right and on the lower sides are affine;
\item if two points on the boundary of the rectangle lie on the same vertical line, their images also lie on the common vertical line;
\item $\overline t_n$ preserves area;
\item $\overline t_n$ is piecewise affine and $L_1$-biLipschitz, where $L_1$ depends only on $L$.
\end{enumerate}

We set
$$
\widetilde Q_j = \widetilde t_n\left(\overline t_n^{-1}(\overline Q_j)\right).
$$

By definition $\widetilde T_n = \bigcup\limits_{Q_j\subset T_n}\widetilde Q_j$.

Since the mappings $\widetilde t_n$ and $\overline t_n$ preserve area,
$$
\mathrm{Area}(\widetilde Q_j) = \mathrm{Area}(\overline Q_j) = A_j.
$$

\smallskip
\noindent{\bf Step 4.} On this step we construct mappings $\widetilde q_{j}:\widetilde\lambda_{j}\cdot B \to\widetilde Q_j.$ 

First we construct mappings
$\overline q_{j}:\widetilde\lambda_{j}\cdot B \to\overline Q_j.$

Recall that $n\in\lbrace1,2,3,4\rbrace$ is a number fixed on step 2.

By definition the boundary of a polygon $\overline Q_j$ consists of a graph of the function $f_{im, k}$, of a union of graphs of functions $f_{i(m+1), l},$ where $(k-1)\cdot2^m \le l \le k\cdot2^{m},$ and two vertical segments lying on the vertical lines $x = (k-1)\cdot2^{-im}$ and $x=k\cdot2^{-im}.$ These functions are piecewise linear and $8C$-Lipschitz by property~\ref{mlp2} of Lemma~\ref{main-lemma}. In each point of the segment $[(k-1)\cdot2^{-im}; k\cdot2^{-im}]$ the difference between the function $f_{im, k}$ and the corresponding function $f_{i(m+1), l}$ is at least $2^{-im}/8C$ by property~\ref{mlp4} of Lemma~\ref{main-lemma}, and in the point $k\cdot2^{-im}$ this difference is at most $2^{-im}$ by property~\ref{mlp3}. So by Lemma~\ref{square-to-graph2} there exists a mapping $\overline q_j: \widetilde\lambda_{j}\cdot B \to\overline Q_j$ such that
\begin{enumerate}[label = \arabic*), ref = \arabic*)]
\item the left (respectively, the right) vertical side of the square is mapped to the left (respectively, to the right) vertical segment of the boundary of the polygon $\overline Q_j$;
\item the upper side of the square is mapped to the graph of the function $f_{im, k}$;
\item the restriction of the mapping $\overline t_n$ on the left side and on the right side of the square are affine;
\item a composition of a restriction of $\overline t_n$ on the lower or on the upper side with the orthogonal projection to the $x$-axis is affine;
\item $\overline q_j$ preserves area;
\item $\overline q_j$ is piecewise affine and $L_2$-biLipschitz, where $L_2$ depends only on $L$.
\end{enumerate}

We define mappings $\widetilde q_j = \widetilde t_n\circ\overline t_n^{-1}\circ\overline q_j.$

\smallskip
\noindent{\bf Step 5.} Finally we check the needed properties of the constructed polygons and mappings.

We set $\widetilde M = L_3\cdot L_1\cdot L_2,$ where $L_3$ bounds from above biLipschitz constants of mappings $\widetilde q_0$ and $\widetilde t_n$ for $n=1,2,3,4.$ Then the mappings $\widetilde q_j$ are $\widetilde M$-biLipschitz.

Mappings $\widetilde q_j$ preserve area and are piecewise affine, since they were defined as a composition of such mappings.

Now we check property~\ref{combinatorial-equivalence}. A composition $\left(\widetilde q_{j}\circ\alpha_{\widetilde\lambda_{j}}\right)^{-1}\circ\left(\widetilde q_{l}\circ\alpha_{\widetilde\lambda_{l}}\right)$ is
\begin{itemize}
\item a relation between empty sets if the polygons $\widetilde Q_j$ and $\widetilde Q_l$ do not intersect;
\item the identity if $j=l;$
\item the mapping from a vertex of a square to a vertex of a square if the polygons $\widetilde Q_j$ and $\widetilde Q_l$ have a common point;
\item an affine mapping from a side of a square to the side of a square in other cases.
\end{itemize}
All these properties are also satisfied for mappings $q_j$ and quadrilaterals $Q_j$ by construction. The equality
$$
\left(\widetilde q_{j}\circ\alpha_{\widetilde\lambda_{j}}\right)^{-1}\circ\left(\widetilde q_{l}\circ\alpha_{\widetilde\lambda_{l}}\right)
=
\left(q_{j}\circ\alpha_{\lambda_{j}}\right)^{-1}\circ\left(q_{l}\circ\alpha_{\lambda_{l}}\right)
$$ 
is satisfied because the tiling into polygons $\widetilde Q_j$ is combinatorially equivalent to the tiling into quadrilaterals $Q_j.$

Property~\ref{interior-cover} is satisfied because initially we divided the square $B$ into the rectangle $\widetilde Q_0$ and four trapezoids each of which is covered by polygons $\widetilde Q_j.$

Property~\ref{boundary-limit} is satisfied because $t_n\circ\alpha_{\sqrt{S_n^{-1}\cdot2^{-m}}}\circ \widetilde t_n^{-1}$ is the identity on a side of the square for all $n=1,2,3,4$ and the diameter of polygons $Q_j$ and $\widetilde Q_j$ converges to zero if $j\to\infty.$
\end{proof}

\begin{proof}[Proof of Proposition~\ref{jacobian-problem}]
By Lemmas~\ref{skew-lemma} and~\ref{trapezoid-lemma} there exists a real number $M'$ and for all $j\ge0$ there exist mappings $q_j':\lambda_j\cdot B\to Q_j$ such that for all $j\ge0$
\begin{itemize}
\item the mappings $q_j$ and $q_j'$ coincide on the boundary of the square;
\item the mapping $q_j'$ is piecewise affine, has a constant Jacobian and is $M'$-biLipschitz.
\end{itemize}

Let $\widetilde q_j:\widetilde\lambda_j\cdot B\to\widetilde Q_j$ be the mappings constructed in Lemma~\ref{half-theorem} for $A_j=\int_{Q_j}f $.

Then a mapping which on each quadrilateral $Q_j$ is defined as
$$
\widetilde q_j\circ\alpha_{\widetilde\lambda_j/\lambda_j}\circ\left(q_j'\right)^{-1}
$$
has the required Jacobian.
\end{proof}

Two mappings $f$ and $g$ are called {\it similar}, if there exist similarities $s_1$ and $s_2$ such that $f = s_1\circ g \circ s_2.$

\begin{lemm}
\label{areafication}
Let $\mathcal P$ be a finite set of similarity classes of piecewise affine mappings from $\partial B$ to the plane and $K$ be a real number. Then there exists a real number $\widetilde K$ such that for any $K$-biLipschitz mapping $\Phi:B\to\mathbb R^2$ such that
$$\left[\Phi\circ q_j\big|_{\lambda_j\cdot\partial B}\right]\in\mathcal P,$$
and mappings $\widetilde q_j$ satisfying all properties listed in Lemma~\ref{half-theorem} for $\widetilde C = K^2$ and
$$A_j = \mathrm{Area}(\Phi(Q_j))$$
there exists an area-preserving $\widetilde K$-biLipschitz mapping $\widetilde\Phi:B\to\mathbb R^2$ such that
\begin{itemize}
\item
$\widetilde\Phi$ is piecewise affine on the interior of the square;
\item
$
\left(\widetilde\Phi\circ\widetilde q_j\right) (\widetilde \lambda_j\cdot z) = 
\left(\Phi\circ q_j\right) ( \lambda_j\cdot z)$
for any $j$ and any $z\in\partial B$;
\item
$\widetilde\Phi(z) = \Phi(z)$ for any $z\in\partial B.$
\end{itemize}
\end{lemm}

\begin{proof}
For each class $k\in\mathcal P$ we choose an arbitrary area-preserving piecewise linear mapping $p_k:B\to\mathbb R^2$ such that $[p_k\big|_{\partial B}] = k$ (for the existence of such mappings see for example~\cite{BD}).

Then for each $j$ there exists $k$ and a similarity $s_j$ such that
\begin{equation}
\label{phi-pk}
\Phi\circ q_j(\lambda_j\cdot z) = s_j\circ p_k (z)
\end{equation}
for any $z\in\partial B.$ Then we define a mapping $\widetilde\Phi$ on each polygon $\widetilde Q_j$
\begin{equation}
\label{Phi-definition}
\widetilde\Phi\big|_{\widetilde Q_{j}} = s_{j}\circ p_k\circ \alpha_{\widetilde\lambda_{j}^{-1}} \circ \widetilde q_{j}^{-1}.
\end{equation}

This mapping is defined correctly on the intersections of polygons $\widetilde Q_j$ by~(\ref{phi-pk}) and property~\ref{combinatorial-equivalence} of Lemma~\ref{half-theorem}.

Now we check that $\widetilde\Phi$ preserves area and is biLipschitz.

By construction $\widetilde \Phi(\widetilde Q_j) = \Phi(Q_j)$ and $\mathrm{Area}\left(\widetilde Q_j\right) = \mathrm{Area}\left(\widetilde \Phi\left(\widetilde Q_j\right)\right).$

The mappings $p_k$ and $\widetilde q_{j}$ preserve area, and other mappings on the right side of equation~(\ref{Phi-definition}) are similarities. Since areas of polygons $\widetilde Q_j$ and $\widetilde \Phi(\widetilde Q_j)$ are equal, the mapping $\widetilde\Phi$ preserve area.

Since each mapping $p_k$ is piecewise affine, preserves area and is an embedding on the boundary, it is actually embedding and hence it is biLipschitz. Since the mappings $p_k$ are finite in number, they all are biLipschitz with a common constant $L'$.

Let $\widetilde K_0 = L'\cdot\widetilde M$, where $\widetilde M$ is given by Lemma~\ref{half-theorem}. Now we prove that the mapping $\widetilde\Phi\big|_{\widetilde Q_j}$ is $\widetilde K_0$-biLipschitz. The mapping $\widetilde q_j$ is $\widetilde M$-biLipschitz. Since the mapping on the left side of the equality~(\ref{Phi-definition}) preserves area, the mappings $p_k$ and $\widetilde q_j$ on the right side preserve area, then contributions to the biLipschitz constant of similarities on the right side cancel. Therefore $\widetilde\Phi\big|_{\widetilde Q_j}$ is $\left(L'\cdot\widetilde M\right)$-biLipschitz.

Since the square $B$ is convex and the mappings $\widetilde\Phi\big|_{\widetilde Q_j}$ are $\widetilde K_0$-Lipschitz, then by property~\ref{interior-cover} of Lemma~\ref{half-theorem} $\widetilde\Phi$ is $\widetilde K_0$-Lipschitz.

From the Lipschitz property it follows that the mapping $\widetilde\Phi$ extends to the boundary of the square by continuity. Now we check that $\widetilde \Phi\big|_{\partial B} = \Phi\big|_{\partial B}.$ Let $z\in\partial B$. If a sequence $z_n\in Q_{j_n}$ converges to $z\in\partial B,$ then since $\widetilde \Phi(\widetilde Q_j) = \Phi(Q_j)$ there exists a sequence $\widetilde z_n\in\widetilde Q_{j_n}$ such that $\Phi(z_n) = \widetilde \Phi(\widetilde z_n).$ By property~\ref{boundary-limit} of Lemma~\ref{half-theorem} the sequence $\widetilde z_n$ converges to $z$. Therefore $\Phi (z) = \widetilde \Phi(z).$

Now we show that the mapping $\widetilde\Phi^{-1}$ is $\left(\max\left(\widetilde K_0, K\right)\right)$-Lipschitz. Let $z_0,z_1$ be interior points of $B,$ $w_i = \widetilde\Phi(z_i)$ if $i=0,1.$ We connect $w_0$ with $w_1$ by a straight line segment $I$ of length $\ell.$ 

If the segment $I$ lies inside the region bounded by the curve $\Phi(\partial B),$ it can be divided into a countably many segments of the form $I\cap\widetilde\Phi(\widetilde Q_j)$ on which the mapping $\widetilde\Phi^{-1}$ is $\widetilde K_0$-Lipschitz. Therefore the preimage of $I$ under $\widetilde \Phi$ is a curve of length at most $\widetilde K_0\cdot\ell.$ So 
$$
\mathrm{dist}(z_0, z_1) \le \widetilde K_0\cdot\mathrm{dist}(w_0, w_1).
$$

If the segment $I$ intersects the curve $\Phi(\partial B),$ let $w_{1/3}$ be the nearest to $w_0$   point of the intersection $I\cap\Phi(\partial B),$ and $w_{2/3}$ be the nearest point to $w_1.$ Let $z_{1/3}\in\partial B$ and $z_{2/3}\in\partial B$ be such points that $\Phi(z_{1/3}) = w_{1/3}$ and $\Phi(z_{2/3}) = w_{2/3}.$ Similarly to the above argument
$$
\mathrm{dist}(z_0, z_{1/3}) \le \widetilde K_0\cdot\mathrm{dist}(w_0, w_{1/3}),\ 
\mathrm{dist}(z_{2/3},z_1) \le \widetilde K_0\cdot\mathrm{dist}(w_{2/3},w_1).
$$

Since $\Phi$ is $K$-biLipschitz
$$
\mathrm{dist}(z_{1/3},z_{2/3}) \le K\cdot\mathrm{dist}(w_{1/3},w_{2/3}).
$$

Summing up the obtained inequalities, we get
$$
\mathrm{dist}(z_0, z_1) \le \max\left(\widetilde K_0, K\right)\cdot\mathrm{dist}(w_0, w_1).
$$

Therefore the mapping $\widetilde \Phi$ is $\widetilde K$-biLipschitz, where $\widetilde K=\max\left(\widetilde K_0, K\right).$
\end{proof}

\begin{proof}[Proof of Proposition~\ref{square-extension-lemma}]
Follows immediately from Lemmas~\ref{Tukia-theorem},~\ref{half-theorem} and~\ref{areafication}.
\end{proof}

\subsection{Merging elements of the tiling}
\label{merge-tiling}
Let $N$ be a positive integer. Let $q_j:\lambda_j\cdot B\to Q_j$ be the mappings from the construction in the beginning of Section~\ref{square-extension} for the given $m.$ Let $q'_j:\lambda'_j\cdot B\to Q'_j$ be mappings from the same construction but for $m' = Nm.$

\begin{lemm}
\label{merging-elements-of-tiling-lemma}
Let $\mathcal P$ be a finite set of similarity classes of piecewise affine embeddings of $\partial B$ to the plane and $N$ be a positive integer. Then there exists a finite set $\mathcal P'$ of similarity classes of piecewise affine mappings from $\partial B$ to the plane such that for any mapping $\Phi:B\to\mathbb R^2$ such that
$$\left[\Phi\circ q_j\big|_{\lambda_j\cdot\partial B}\right]\in\mathcal P$$
the following is true:
$$
\left[\Phi\circ q_j'\big|_{\lambda_j'\cdot\partial B}\right]\in\mathcal P'.
$$
\end{lemm}

\begin{proof}
In each class $k\in\mathcal P$ we fix a representative $p_k.$

\begin{figure}[h]
\center
\includegraphics{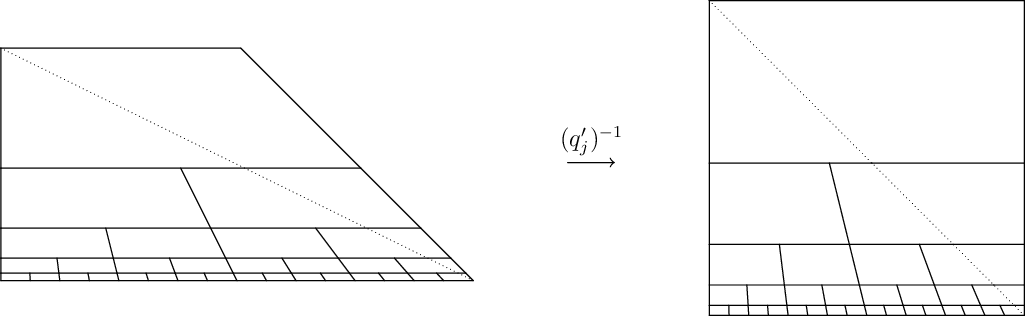}
\caption{The preimage of the tiling of a quadrilateral $Q_j'$ into quadrilaterals $Q_l$ under $q_j'.$ The case $m=1,$ $N=5.$}
\label{m-multiplying}
\end{figure}

Each quadrilateral $Q'_j$ is tiled into quadrilaterals $Q_l.$ 
If $Q_{l_0}$ and $Q_{l_1}$ have a common side, then for each $p_{k_0}$ there exists only a finite number of similarities $s$ such that
$$
p_{k_0}\circ q_{l_0}^{-1}\big|_{Q_{l_0}\cap Q_{l_1}} = s\circ p_{k_1}\circ q_{l_1}^{-1}\big|_{Q_{l_0}\cap Q_{l_1}}
$$
for some $p_{k_1}.$ Repeating this argument for all adjacent quadrilaterals $Q_l$ contained in $Q'_j$ we obtain that the set of mappings
$$F:\bigcup\limits_{Q_l\subset Q_j'}\left(q'_j\right)^{-1}(\partial  Q_l)\to\mathbb R^2$$
such that
$$\left[F\circ \left(q'_j\right)^{-1}\circ q_l\big|_{\lambda_l\cdot\partial B}\right]\in\mathcal P
\text{ if }
Q_l\subset Q'_j
$$
is a union of finite number of similarity classes. We denote this set of similarity classes by $\mathcal P_j$ and let
$$
\partial\mathcal P_j
=
\left\{
F\big|_{\lambda_j'\cdot\partial B}\left|\ F:\bigcup\limits_{Q_l\subset Q_j'}\left(q'_j\right)^{-1}(\partial  Q_l)\to\mathbb R^2,\ [F]\in \mathcal P_j
\right.\right\}.
$$

If $j_0,j_1\neq 0$  the tilings of quadrilaterals $Q_{j_0}'$ and $Q_{j_1}'$ are combinatorially equivalent (Figure~\ref{m-multiplying}). Also the corresponding sides of quadrilaterals $Q_{j_0}'$ and $Q_{j_1}'$ for $j\neq0$ are divided into segments in the same ratios. Since mappings $q_j'$ and $q_l$ are affine on each side of the square, the corresponding sides of squares $\lambda_{j_0}'\cdot B$ and $\lambda_{j_1}'\cdot B$ are divided in the same ratios. Therefore, if the mappings $q_{j_0}'$ and $q_{j_1}'$ map the same side of the square (respectively the upper, lower, right or left) to the longer base of a trapezoid, then
$$
\partial \mathcal P_{j_0} = \partial \mathcal P_{j_1}.
$$

So $\mathcal P' = \bigcup\limits_j\partial\mathcal P_j$ is finite.

\end{proof}

\subsection{Correcting the mapping on the square $r_0\cdot B$}
\label{correction-on-some-square}

This section is devoted to the proof of the following proposition.

\begin{prop}
\label{extension-to-annulus-lemma}
For any $r_0,$ $r_1,$ $L,$ where $0<r_0<r_1<1,$ there exists such $\widetilde K$ that any $L$-biLipschitz mapping $\phi:\partial B\to\mathbb R^2$, such that the curve $\phi(\partial B)$ bounds a region of area 4 and contains a square $r_1\cdot B$, can be extended to an area-preserving locally piecewise affine on the interior of the square $\widetilde K$-biLipschitz mapping of $B$ which is the identity on $r_0\cdot B.$
\end{prop}

Let $\Phi$ be a $K$-biLipschitz extension of $\phi$ provided by Lemma~\ref{Tukia-theorem}.
Let $r_1' = \frac{r_0+r_1}2.$ Let $N$ be such positive integer that
$$
K\cdot 2^{-Nm} < \frac{r_1-r_0}2
\quad\text{and}\quad
K^2\cdot 4\cdot 2^{-Nm} < 2 - 2r_1'.
$$

We can assume that $N$ depends only on $L$, $r_0$ and $r_1$ because $m = m(L)$ and $K = K(L).$ Then
$$
\Phi\left(\left(1-2^{-Nm}\right)\cdot B\right)\supset r_1'\cdot B 
\quad
\text{ and }
\quad
\mathrm{Area}\left(\Phi\left(\left(1-2^{-Nm}\right)\cdot B\right)\right) > 2\cdot (1+r_1').
$$

Now we apply Lemma~\ref{merging-elements-of-tiling-lemma} for $m'=Nm.$ The set $\mathcal P$ of similarity classes of mappings $\Phi\circ q_j'\big|_{\lambda_j'\cdot\partial B}$ for all possible $j$ is finite and this set depends only on $L$ and $N.$ Therefore by Lemma~\ref{areafication} there exists an area-preserving locally piecewise affine on the interior of the square $\widetilde K(L,N)$-biLipschitz mapping $\widetilde\Phi:B\to\Phi(B)$ such that
\begin{itemize}
\item
$\widetilde\Phi\big|_{\partial B} = \phi;$
\item
$\widetilde\Phi\left(\widetilde Q_j\right) = \Phi(Q_j)$ for any $j$;
\item
$\left[\widetilde\Phi\circ \widetilde q_j\big|_{\widetilde\lambda_j\cdot \partial B}\right]\in\mathcal P$ for any $j,$
\end{itemize}
where the mappings $\widetilde q_j:\widetilde\lambda_j\cdot B\to\widetilde Q_j$ satisfy all properties listed in the statement of Lemma~\ref{half-theorem} for the tiling into quadrilaterals $Q_j'.$ We also recall that by Lemma~\ref{half-theorem} the mapping $\widetilde q_0$ is affine and that $\widetilde Q_0$ is a rectangle whose sides are parallel to the sides of $B.$

By construction
$$\widetilde Q_0\supset r_1'\cdot B + \widetilde O,$$
where $\widetilde O$ is the center of the rectangle $\widetilde Q_0,$ because $\mathrm{Area}\left(\widetilde Q_0\right) = \mathrm{Area}(\Phi(Q_0))>2\cdot (1+r_1').$ We also have
$$
\widetilde\Phi\left(\widetilde Q_0\right)\supset r_1'\cdot B.
$$

\begin{lemm}
\label{square-travel}
Let $P$ be a polygon in the plane, $\mathrm{Area}(P)=4,$ $f:\partial B\to\partial P$ be a piecewise linear homeomorphism, $r_0,r_1$ be such real numbers that $0<r_0<r_1<1,$ 
$C_0 \ge 1.$ Then there exists such $L'$ that for any $s\in[r_1;C_0]$ and for any isometric embedding 
$b:r_1\cdot B\to s\cdot P$ the embedding $b\big|_{r_0\cdot B}$ can be extended to an area-preserving piecewise affine $L'$-biLipschitz mapping $F:s\cdot 
B\to s\cdot P$ such that $F(s\cdot z) = s\cdot f(z)$ for any $z\in\partial B.$
\end{lemm}

\begin{proof}

Let $\mathcal B$ be a set of all pairs $(b,s),$ where $b:\mathbb R^2\to\mathbb R^2$ is an isometry, $b(r_1\cdot B)\subset s\cdot P$ and $s\in[r_1;C_0].$ We equip this set with the metric
$$
\mathrm{dist}\left((b,s),(b',s')\right) = \mathrm{dist}\left(b\big|_{r_1\cdot B},b'\big|_{r_1\cdot B}\right) + \left|s-s'\right|.
$$

It is clear that $\mathcal B$ is compact.

For any pair $(b,s)\in\mathcal B$ we fix an area-preserving piecewise affine mapping $F_{b,s}:s\cdot B\to s\cdot P$ such that
$$F_{b,s}\big|_{\frac{2r_0+r_1}3\cdot B} = b\big|_{\frac{2r_0+r_1}3\cdot B}
\quad\text{and}\quad
F_{b,s}(s\cdot z) = s\cdot f(z)
\text{ for any }
z\in\partial B.$$

By Lemma~\ref{square-rotation} for any isometry $b:\mathbb R^2\to\mathbb R^2$ there exists such positive $\varepsilon = \varepsilon(b)$ that if $\mathrm{dist}(b\big|_{r_1\cdot B},b'\big|_{r_1\cdot B})<\varepsilon$ then there exists an area-preserving piecewise affine 4-biLipschitz mapping $\Phi_{b,b'}:\mathbb R^2\to\mathbb R^2$ such that
$$
\Phi_{b,b'}\circ b\big|_{\frac{2r_0+r_1}3\cdot B} = b'\big|_{\frac{2r_0+r_1}3\cdot B},\quad
\Phi_{b,b'}\big|_{\mathbb R^2\setminus b\left(\frac{r_0+2r_1}3\cdot B\right)} = \mathrm{id}_{\mathbb R^2\setminus b\left(\frac{r_0+2r_1}3\cdot B\right)}
$$
(if $b$ and $b'$ differ by a parallel translation, we should apply Lemma~\ref{square-rotation} twice using two rotations by opposite angles around distinct points). For any isometry $b:\mathbb R^2\to\mathbb R^2$ and $\lambda > 0$ we denote by $b_{\lambda}$ an isometry $\alpha_{\lambda}\circ b\circ\alpha_{\lambda^{-1}},$ where $\alpha_{\lambda}(z) = \lambda\cdot z$ for any $z\in\mathbb R^2.$ If $(b,s)\in\mathcal B$ then we set
$$
\mathcal U_{b,s} = \left\{(b',s')\in\mathcal B:\ \mathrm{dist}\left(b\big|_{r_1\cdot B}, b'_{s/s'}\big|_{r_1\cdot B} \right)<\varepsilon(b)\text{\quad and \quad} s' > s\cdot\max\left(\frac{r_0+2r_1}{3r_1},\frac{3r_0}{2r_0+r_1} \right)  \right\}
$$
(actually among two real numbers in the maximum function the first one is always greater than the second). By definition $\mathcal U_{b,s}$ is open and $(b,s)\in\mathcal U_{b,s}.$ If $(b',s')\in\mathcal U_{b,s},$ then $\frac{s'}{s}\cdot r_1 > \frac{r_0+2r_1}3,$ therefore
$$b'_{s/s'}\left(\frac{r_0+2r_1}3\cdot B\right)\subset s\cdot P
\quad\text{and}\quad
\Phi_{b, b'_{s/s'}}\big|_{\mathbb R^2\setminus s\cdot P} = \mathrm{id}_{\mathbb R^2\setminus s\cdot P}.
$$

If $(b',s')\in \mathcal U_{b,s},$ then $\frac{s}{s'}\cdot r_0 < \frac{2r_0+r_1}3,$
$$
\alpha_{s'/s}\circ \Phi_{b,b'_{s/s'}}\circ F_{b,s}\circ\alpha_{s/s'}\big|_{r_0\cdot B} = b'\big|_{r_0\cdot B}
\quad\text{and}\quad
\alpha_{s'/s}\circ \Phi_{b,b'_{s/s'}}\circ F_{b,s}\circ\alpha_{s/s'}(s'\cdot z) = s'\cdot f(z)
$$
for any $z\in\partial B.$

Since $\mathcal B$ is compact, there exist a finite number of neighborhoods $\mathcal U_{b_i,s_i}$ which cover $\mathcal B.$ Hence for any $(b',s')\in\mathcal B$ there exists such $i$ that a mapping
$$
F' = \alpha_{s'/s_i}\circ \Phi_{b_i,b'_{s_i/s'}}\circ F_{b_i,s_i}\circ\alpha_{s_i/s'}
$$
is piecewise affine, preserves area, is $L'$-biLipschitz with $L' = 4\max\limits_i\mathrm{bilip}(F_{b_i,s_i}),$ and
\begin{itemize}
\item $F'\big|_{r_0\cdot B} = b'\big|_{r_0\cdot B};$
\item $F'(s'\cdot z) = s'\cdot f(z)$ for any $z\in\partial B.$
\end{itemize}
\end{proof}

We fix $M(r_0,r_1)$-biLipschitz area-preserving piecewise affine mapping $\Psi_0:\mathbb R^2\to \mathbb R^2$ the restriction of which on the complement to the square $\frac{r_0+2r_1}3\cdot B$ is the identity, and the restriction on the square $\frac{2r_0+r_1}3\cdot B$ is a rotation by $90^{\circ}.$

Let $\lambda\in [1/M_1;M_1],$ where $M_1 = \min\left(\frac{2r_0+r_1}{3r_0}, \frac{3r_1}{r_0+2r_1} \right) = \frac{3r_1}{r_0+2r_1}.$ Let $A_{\lambda}:(x,y)\mapsto (\lambda\cdot x, y/\lambda).$ Then a mapping $A_{\lambda}\circ\Psi_0\circ A_{\lambda}$ is the identity on the square $r_0\cdot B,$ and coincides with $A_{\lambda^2}$ outside the square $r_1\cdot B.$ Its biLipschitz constant is bounded above by a function of $r_0$ and $r_1.$

\begin{lemm}\label{square-to-rectangle}
Let $C_1\ge 1$. Let $s$ and $k$ be such real numbers that $s/k,$ $s\cdot k\in[r_1;1]$ and $k\in[1/C_1;C_1].$ There exists $M(C_1,r_0,r_1)$-biLipschitz area-preserving piecewise affine mapping $\Psi: s\cdot B\to [-s/k;s/k]\times[-s\cdot k;s\cdot k]$ such that
\begin{itemize}
\item $\Psi\big|_{r_0\cdot B}$ is the identity;
\item $\Psi(x,y) = (x/k, k\cdot y)$ if $(x,y)\in s\cdot \partial B.$
\end{itemize}
\end{lemm}

\begin{proof}
For any $s$ and $k$ we choose a mapping $\Psi_{s,k}$ which satisfies all conditions in the statement of the lemma except the estimate on the biLipschitz constant. Since $\Psi_{s,k}$ can be extended by an affine mapping on the complement to the square $s\cdot B,$ any mapping $\Psi_{s_0, k}$ satisfies all conditions for $s>s_0.$ Hence we may assume that $\Psi_{s, k} = \Psi_{s(k), k}$ for all admissible $s$ and $k$, where $s(k)$ is the minimal admissible $s$ for a fixed $k$. By compactness there exist $k_1,\dots,k_n\in[1/C_1;C_1]$ such that for any $k\in[1/C_1;C_1]$ there exist such $\lambda\in [1/M_1;M_1]$ and $i\in\{1,2,\dots,n\}$ that $k = \lambda^2\cdot k_i.$ Finally let
$$
\Psi=\ A_{\lambda}\circ\Psi_0\circ A_{\lambda}\circ \Psi_{s(k_i), k_i}.
$$

Then $\Psi$ is the required mapping. Its biLipschitz constant is bounded above by a function of $r_0, r_1$ and $C_1.$
\end{proof}

Recall that $\left[\widetilde\Phi\circ\widetilde q_0\big|_{\widetilde\lambda_0\cdot\partial B}\right] \in\mathcal P$, where $\mathcal P$ is finite and depends only on $r_0,r_1$ and $L$. Since $\widetilde\lambda_0<1$ and $\widetilde\Phi\left(\widetilde Q_0\right)\supset r_1'\cdot B,$ then by Lemma~\ref{square-travel} there exists an area-preserving piecewise affine $L'(r_0,r_1,L)$-biLipschitz mapping $F:\widetilde\lambda_0\cdot B\to\widetilde\Phi\left(\widetilde Q_0\right)$ such that
\begin{itemize}
\item $F\big|_{r_0\cdot B} = \mathrm{id}_{r_0\cdot B}$;
\item $F\big|_{\widetilde\lambda_0\cdot\partial B} = \widetilde\Phi\circ\widetilde q_0\big|_{\widetilde\lambda_0\cdot\partial B}.$
\end{itemize}

By Lemma~\ref{square-to-rectangle} there exists a mapping $G:\widetilde\lambda_0\cdot B \to \widetilde Q_0$ such that
\begin{itemize}
\item $G\big|_{r_0\cdot B}$ is a parallel translation by a vector $\overrightarrow{O\widetilde O},$ where $O$ is a coordinate origin;
\item $G (x, y) = (x/k, k\cdot y) + \widetilde O$ for any point $(x,y)\in\widetilde\lambda_0\cdot\partial B,$ where $k$ is a positive constant;
\item $G$ is piecewise affine, preserves area, and its biLipschitz constant is bounded above by a function of $r_0, r_1$ and $L$ (since the lengths of sides of the rectangle $\widetilde Q_0$ are bounded from below by a positive function of $L$).
\end{itemize}

By construction the mapping $G^{-1}\circ\widetilde q_0\big|_{\widetilde\lambda_0\cdot\partial B}$ is an isometry. Let $b:\mathbb R^2\to\mathbb R^2$ be an extension of this isometry. Then
$$
G\circ b\big|_{\lambda_0\cdot \partial B} = \widetilde q_0\big|_{\widetilde\lambda_0\cdot B}.
$$

Since $B\supset r_1'\cdot B  + \widetilde O$, then by Lemma~\ref{square-travel} there exists an area-preserving piecewise affine mapping $H:B\to B$ whose biLipschitz constant is bounded from above by a function of $r_0$ and $r_1$ such that
\begin{itemize}
\item $H(z) = b(z) + \widetilde O$ for any $z\in r_0\cdot B;$
\item $H\big|_{\partial B} = \mathrm{id}\big|_{\partial B}.$
\end{itemize}

Finally we define a mapping $\Phi_{\square}:B\to \widetilde \Phi(B)$ by the rule
$$
\Phi_{\square}(z) = 
\left\{
\begin{aligned}
F \circ b^{-1} \circ G^{-1} \circ H(z),& \text{ if } z\in H^{-1}\left(\widetilde Q_0\right),\\
\widetilde\Phi(z),& \text{ if } z\notin H^{-1}\left(\widetilde Q_0\right).
\end{aligned}
\right.
$$

Then $\Phi_{\square}\big|_{r_0\cdot B} = \mathrm{id}_{r_0\cdot B}$ and $\Phi_{\square}\big|_{\partial B} = \phi.$ Also the mapping $\Phi_{\square}$ is locally piecewise affine on the interior of the square, preserves area and its biLipschitz constant is bounded from above by a function of $r_0, r_1$ and $L.$ So $\Phi_{\square}$ is the desired mapping.

\subsection{Extending to the complement}
\label{extension-to-complement}
In this section we prove Proposition~\ref{extension-to-complement-lemma} and we finish the proof of Theorems~\ref{main-theorem},~\ref{annulus-version} and~\ref{the-technical-lemma-without-parameters}.

For any $R$ and $r$ such that $R>r$ we denote by $\Omega_{R,r}$ the annulus bounded by squares $\alpha_R(B)$ and $\alpha_r(B).$

\begin{lemm}
\label{infinity-extension-lemma}
Let $r_1, R_0,R_1$ be real numbers  and let $r_1<1<R_1<R_0$. Let $\phi:\partial B\to\mathbb R^2$ be an $L$-biLipschitz mapping with image lying in $\Omega_{R_1,r_1}$ and bounding a region of area 4. Then $\phi$
can be extended to a mapping $\widetilde\Phi$ of the complement to the interior of $B$ and this mapping
\begin{itemize}
\item is locally piecewise affine in the complement of $B$;
\item preserves area;
\item is the identity on $\mathbb R^2\setminus \left(R_0\cdot B\right)$;
\item is $\widetilde K$-biLipschitz, where $\widetilde K$ is some function of $L,r_1,R_0$ and $R_1$.
\end{itemize}
\end{lemm}

\begin{rema}
At first glance the condition that the curve bounds a region of area 4 is unnecessary because $\widetilde\Phi$ is a mapping between unbounded components. Actually without this condition we can not make $\widetilde \Phi$ to be the identity outside some square.
\end{rema}

\begin{proof}
Let us fix an arbitrary piecewise affine mapping $\delta: \Omega_{R_1,r_1} \to \Omega_{R_1,r_1}$ which preserves area and orientation and exchanges the boundary components. The mapping $\delta$ depends only on $r_1$ and $R_1$, hence the biLipschitz constant of a mapping $\delta\circ\phi$ is bounded above by a function of $L,$ $r_1$ and $R_1.$ Let the curve $\delta\circ\phi(\partial B)$ bound a region of total area $S$. Then
\begin{equation}
\label{Sdef}
4 + S = 4\left(r_1^2 + R_1^2\right).
\end{equation}

Let $r_0$ be such real number that
\begin{equation}
\label{r0def}
\mathrm{Area}(\Omega_{r_1,r_0}) = 4\min\left(r_1^2/2, R_0^2 - R_1^2\right).
\end{equation}

By Lemma~\ref{extension-to-annulus-lemma} for a mapping $\alpha_{2\sqrt{S}^{-1}}\circ\delta\circ\phi$ and real numbers $\frac{2r_0}{\sqrt{S}},$ $\frac{2r_1}{\sqrt{S}}$ there exists a mapping $\Psi: B\to\mathbb R^2$ such that
\begin{itemize}
\item $\Psi\big|_{\partial B} = \alpha_{2\sqrt{S}^{-1}}\circ\delta\circ\phi;$
\item $\Psi\big|_{\frac{2r_0}{\sqrt{S}}\cdot B} = \mathrm{id}_{\frac{2r_0}{\sqrt{S}}\cdot B};$
\item $\Psi$ is locally piecewise affine on the interior of the square, preserves area and its biLipschitz constant is bounded above by a function of $L,$ $r_1,$ $R_0,$ and $R_1.$
\end{itemize}

A mapping $\alpha_{\sqrt{S}/2}\circ\Psi\circ\alpha_{2\sqrt{S}^{-1}}$ we denote by $\Psi_{\alpha}.$ Then
$$\Psi_{\alpha}\big|_{\sqrt{S}/2\cdot\partial B} = \delta\circ\phi\circ\alpha_{2\sqrt{S}^{-1}}
\quad\text{and}\quad
\Psi_{\alpha}\big|_{r_0\cdot B} = \mathrm{id}_{r_0\cdot B}.
$$

Since $\sqrt{S}/2 > r_1 > r_0,$ there exists such real number $R_0'$ that
$$\mathrm{Area}(\Omega_{R_0', 1}) = \mathrm{Area}(\Omega_{\sqrt {S}/2, r_0}).$$

Then by (\ref{Sdef}) and (\ref{r0def})
\begin{equation}
\label{R0prime}
\left(R_0'\right)^2 = 1 + S/4 - r_0^2 = r_1^2 + R_1^2 - r_0^2 = \left(r_1^2 -r_0^2\right) + R_1^2 = \min\left(r_1^2/2, R_0^2 - R_1^2\right) + R_1^2 \le R_0^2.
\end{equation}

Let $\widetilde \delta: \Omega_{R_0', 1} \to \Omega_{\sqrt {S}/2, r_0}$ be such area-preserving piecewise affine mapping that
$$
\widetilde \delta\big|_{R_0'\cdot\partial B} = \alpha_{r_0/R_0'}\big|_{R_0'\cdot\partial B}
\quad\text{and}\quad
\widetilde \delta\big|_{\partial B} = \alpha_{\sqrt S/2}\big|_{\partial B}.
$$


Since $\left(R_0'\right)^2 = \left(r_1^2 -r_0^2\right) + R_1^2$ by~(\ref{R0prime}), we have
$$
\mathrm{Area}(\Omega_{r_1,r_0}) = \mathrm{Area}(\Omega_{R_0',R_1}).
$$

Therefore we can extend the mapping $\delta$ to the annulus $\Omega_{r_1,r_0}$ such that it preserves area, is piecewise affine and its restriction to $r_0\cdot B$ coincides with a homothety $\alpha_{R_0'/r_0}.$

We define a mapping $\widetilde\Phi$ on the annulus $\Omega_{R_0',1}$ as a composition
$$
\Omega_{R_0',1}
\xrightarrow{\widetilde\delta}
\Omega_{\sqrt {S}/2, r_0}
\xrightarrow{\Psi_{\alpha}}
\Omega_{R_1, r_0}
\xrightarrow{\delta^{-1}}
\Omega_{R_0',r_1}.
$$

By construction $\widetilde \Phi\big|_{\partial B} = \phi.$ The mapping $\widetilde\Phi$ preserves area, since the mappings $\widetilde\delta,$ $\Psi_{\alpha}$ and $\delta$ preserve area. By construction the restriction $\widetilde\Phi\big|_{R_0'\cdot\partial B}:R_0'\cdot\partial B\to R_0'\cdot\partial B$ is a composition of a homotheties with center in the coordinate origin, hence it is the identity. Therefore we can extend the mapping $\widetilde \Phi$ by the identity to the complement to the square $R_0'\cdot B.$ Since $R_0'\le R_0$ by inequality~(\ref{R0prime}), $\widetilde\Phi$ is the identity outside the square $R_0\cdot B.$
\end{proof}

Let $O$ be the coordinate origin.

\begin{prop}
\label{extension-to-complement-lemma}
For any $L$ there exist $\widetilde K$ and $R_0$ such that any $L$-biLipschitz mapping $\phi:\partial B\to\mathbb R^2$ onto the boundary of a region of area 4 can be extended to a mapping of the complement to the interior of $B$ such that this mapping
\begin{itemize}
\item is locally piecewise affine on the complement to $B$;
\item preserves area;
\item $\widetilde K$-biLipschitz;
\item is a parallel translation by a vector $\overrightarrow{OW}$ outside the square $R_0\cdot B,$ where $W$ is a point lying in the bounded component of the complement to the curve $\phi(\partial B).$
\end{itemize}
\end{prop}

\begin{proof}
By Lemma~\ref{Tukia-theorem} $\phi$ can be extended to a $K_0$-biLipschitz mapping of the square. Let $W$ be the image of the origin under this mapping. Then
\begin{itemize}
\item
$1/K_0 \le \mathrm{dist}(z, W) \le \sqrt2\cdot K_0$ for any $z\in\phi(\partial B);$
\item
$W$ lies inside the bounded component of the complement to the curve $\phi(\partial B).$
\end{itemize}

Consider a parallel translation which maps $W$ to the origin. In the following let us  denote the composition of $\phi$ with this translation also by $\phi.$ Then the curve $\phi(\partial B)$ lies inside the annulus $\Omega_{\sqrt{2}\cdot K_0, 1/\left(\sqrt{2}\cdot K_0\right)}.$

We apply Lemma~\ref{infinity-extension-lemma} to the mapping $\phi$ and real numbers $r_1 = 1/\left(\sqrt{2}\cdot K_0\right),$ $R_1 = \sqrt{2}\cdot K_0$ and $R_0 = R_1 + 1.$ Since $K_0$ depends only on $L$, we are done.

\end{proof}

\begin{proof}[Proof of Theorem~\ref{main-theorem}]
Immediately follows from Propositions~\ref{square-extension-lemma} and~\ref{extension-to-complement-lemma}.
\end{proof}

\begin{proof}[Proof of Theorem~\ref{annulus-version}]
First we apply Propositions~\ref{extension-to-annulus-lemma} and~\ref{extension-to-complement-lemma}. We obtain a mapping $\widetilde \Phi$ which satisfies all conditions in the statement of the theorem except that it is not the identity outside a square $R\cdot B,$ where $R = R(L)$ is some real number. Outside a square $R_0\cdot B$, where $R_0 = R_0(L)$, it is a parallel translation by a vector $\overrightarrow{OW},$ where $W$ is some point inside a region bounded by the curve $\phi(\partial B).$

Since the points $O$ and $W$ lie inside the region bounded by $\phi(\partial B),$ then
$$
\mathrm{dist}(O,W) \le \mathrm{diam}(\phi(\partial B)).
$$

Since $\phi$ is $L$-biLipschitz,
$$
\mathrm{diam}(\phi(\partial B))\le 2\sqrt2\cdot L.
$$

Let $R = 2\sqrt2\cdot L + 2R_0.$

Now we apply Lemma~\ref{square-travel} for $C_0 = 1,$ $P = B,$ $f = \mathrm{id},$ $r_0 = R_0/R,$ $r_1 = 2R_0/R,$ $b$ which is a parallel translation by a vector $\overrightarrow{OW}/R.$ Let $F$ be the obtained mapping. Let $s$ be a parallel translation by a vector $\overrightarrow{OW}.$ Then the mapping
$$
s \circ\left(\alpha_R\circ F^{-1}\circ\alpha_{R^{-1}} \right)\circ\widetilde\Phi
$$
is the one sought for.
\end{proof}

\begin{proof}[Proof of Theorem~\ref{the-technical-lemma-without-parameters}]
Immediately follows from Proposition~\ref{extension-to-annulus-lemma} and Lemma~\ref{infinity-extension-lemma}.
\end{proof}

\section{Main lemma}
\label{main-lemma-section}

In the previous section we called a pair of integers $(N, k)$ {\it admissible} if $N\ge 0$ and $k\in\left\lbrace1, 2, \dots, 2^N\right\rbrace.$

\begin{lemm}
\label{main-lemma}
Let $C>0$ and assume that for every admissible pair $(N, k)$ a real number $A_{N, k}$ is given such that $C^{-1}\cdot 2^{-2N} \le A_{N, k} \le C\cdot 2^{-2N}.$ Then there exist functions $f_{N, k}$ such that for any admissible pair $(N, k)$
\begin{enumerate}[label = \arabic*), ref = \arabic*)]
\item
\label{mlp1}
the function $f_{N, k}$ is defined on a segment $\left[(k-1)\cdot2^{-N}; k\cdot2^{-N}\right]$ and is positive;
\item
\label{mlp2}
the function $f_{N, k}$ is piecewise linear, and its slope can only be zero or $\pm\varkappa$, where $\varkappa = 8C$;
\item
\label{mlp3}
$f_{N, k}\left((k-1)\cdot2^{-N}\right) = f_{N, k}\left(k\cdot2^{-N}\right) = 2^{-N};$
\item
\label{mlp4}
$f_{N, k} - f_{N+1, 2k-1} \ge d\cdot 2^{-N}$ and $f_{N, k} - f_{N+1, 2k} \ge d\cdot 2^{-N},$ where $d = 1/(8C)$, at each point where both functions are defined;
\item
\label{mlp5}
$$
\int f_{N, k}(x) dx - \int f_{N+1, 2k-1}(x) dx - \int f_{N+1, 2k}(x) dx = A_{N, k},
$$
where the integral is over the domain of a function.

\end{enumerate}
\end{lemm}

\begin{figure}[h]
\center
\includegraphics{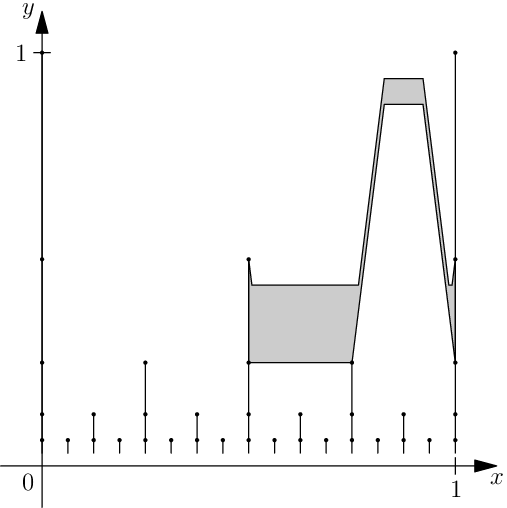}
\caption{A polygon bounded by graphs of three functions $f_{N,k}$ and two vertical segments.}
\label{skeleton-with-graphs}
\end{figure}

In Figure~\ref{skeleton-with-graphs} we illustrate conditions in the statement of Lemma~\ref{main-lemma}. We marked points of the form $\left(k\cdot2^{-N},2^{-N}\right)$ for any admissible pair $(N, k)$, and of the form $\left(0, 2^{-N}\right)$ for all non-negative integer $N.$ Points lying on a common vertical line are connected by a segment. Condition~\ref{mlp3} means that the endpoints of graphs of functions $f_{N,k}$ are marked. Graphs of functions $f_{N, k}$, $f_{N+1, 2k-1},$ $f_{N+1, 2k}$ together with two vertical segments form a closed simple curve. Condition~\ref{mlp5} means that this curve bounds a region of area $A_{N,k}.$ Conditions~\ref{mlp2} and~\ref{mlp4} will allow us to bound the biLipschitz constant of a mapping from this region to a square.

From conditions~\ref{mlp1},~\ref{mlp2} and~\ref{mlp3} it follows that $f_{N,k}\le (1+\varkappa)\cdot2^{-N}$ on a domain of the function. Therefore
\begin{equation}
\label{area-to-zero}
\lim\limits_{N\to+\infty}\sum\limits_{1\le k\le 2^N}\int f_{N,k}(x) dx = 0.
\end{equation}

For each admissible pair $(N, k)$ we set $S_{N,k} = \sum A_{M, l},$ where we sum over all admissible pairs $(M, l)$ such that $(k-1)\cdot 2^{-N}<l\cdot 2^{-M}\le k\cdot 2^{-N}$ and $M\ge N.$ From condition~\ref{mlp5} and equation~(\ref{area-to-zero}) it follows that $S_{N,k}$ is the area below the graph of the function $f_{N,k}.$ Since $C^{-1}\cdot 2^{-2M} \le A_{M, l} \le C\cdot 2^{-2M},$

\begin{equation}
\label{S-inequality}
2C^{-1}\cdot 2^{-2N} \le S_{N, k} \le 2C\cdot 2^{-2N}.
\end{equation}

First we will construct auxiliary functions $f_{N,k}^{i}$ for each non-negative integer~$i$. Functions $f_{N,k}$ will be obtained as a limit

$$
f_{N,k} = \lim\limits_{i\to+\infty} f_{N,k}^{i}.
$$

We will construct functions $f_{N,k}^i$ by induction. For each admissible pair we will define a number $T(N, k)$ so that $T$ is a bijection between the set of all admissible pairs and the set of all non-negative integers and that $f_{N,k}^i = f_{N,k}^j$ if $i\ge j \ge T(N,k).$ We call number $T(N, k)$ {\it the moment of stabilization} of the corresponding function. If $i\ge T(N, k)$, we call a function $f_{N,k}^i$ {\it stabilized}.

For each admissible $(M, l)$ we denote a segment $\left[(l-1)\cdot2^{-M}; l\cdot 2^{-M}\right]$ by $I(M, l).$

For any $x_0\in\mathbb R$ we define a function $h_{x_0} (x) = \varkappa\cdot|x-x_0|.$

\begin{figure}[h]
\center
\includegraphics{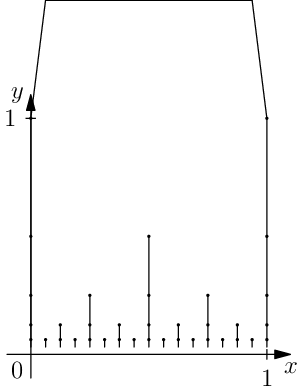}
\includegraphics{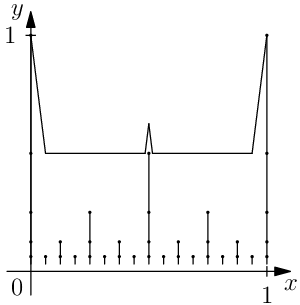}
\includegraphics{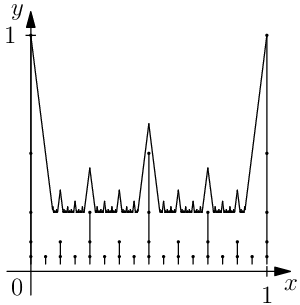}
\caption{Auxiliary function $g$ from Definition~\ref{minimal-function} for $y_0 = 3/2, 1/2$ and for $y_0\le 1/4$ (from the left to the right) and for parameters $N=0,$ $k=1,$ $\varkappa = 8,$ $d=1/8$.}
\label{aux-functions}
\end{figure}

\begin{defi}
\label{minimal-function}
Let $(N,k)$ be admissible and let $y_0$ be a non-negative integer. We define a function $g = g_{N, k, y_0}$ on the segment $I(N,k)$ as follows (Figure~\ref{aux-functions}):

If $y_0 > 2^{-N}$, then 
$$
g = \min\left(y_0, 2^{-N} + h_{(k-1)\cdot2^{-N}}, 2^{-N} + h_{k\cdot2^{-N}}\right).
$$

If $y_0 \le 2^{-N}$, then
$$
g = \max\left(y_0,  \max\limits_{l\cdot2^{-M}\in I(N,k)}\left(2^{-M} + 2d\cdot\left(2^{-N} - 2^{-M}\right) -h_{l\cdot2^{-M}}\right)\right).
$$
\end{defi}

\begin{clai}
\label{piecewise-linear-g}
If $y_0>2d\cdot2^{-N}$ the function $g$ is piecewise linear.
\end{clai}
\begin{proof}
This follows from the fact that only a finite number of functions in the maximum function are greater than $y_0$ and that all these functions are piecewise linear.
\end{proof}

\begin{clai}
\label{g-integral}
If $y_0 = 2d\cdot2^{-N}$, then
$$
\int_{I(N,k)} g(x)dx < 2^{-2N}\cdot\left(2d + \frac3{2\varkappa}\right).
$$
\end{clai}

\begin{proof}
By definition
$$
g = \max\left(2d\cdot2^{-N},  \max\limits_{l\cdot2^{-M}\in I(N,k)}\left(2^{-M} + 2d\cdot\left(2^{-N} - 2^{-M}\right) -h_{l\cdot2^{-M}}\right)\right).
$$

Then
$$
g =2d\cdot2^{-N} + \max\left(0,  \max\limits_{l\cdot2^{-M}\in I(N,k)}\left((1-2d)\cdot2^{-M} -h_{l\cdot2^{-M}}\right)\right).
$$

Then
$$
\int_{I(N,k)} g(x)dx \le 2d\cdot2^{-2N} + \sum\limits_{l\cdot2^{-M}\in I(N,k)}\int_{I(N,k)}\max\left(0, (1-2d)\cdot2^{-M} -h_{l\cdot2^{-M}}(x)\right)dx.
$$

The following integral equals the area of an isosceles triangle such that its height has length $(1-2d)\cdot2^{-M}$ and the slope of its legs equal $\varkappa$:

$$
\int_{\mathbb R}\max\left(0, (1-2d)\cdot2^{-M} -h_{l\cdot2^{-M}}(x)\right)dx = (1-2d)^2\cdot2^{-2M}/\varkappa.
$$

If $M>N$ the number of triangles equals $2^{M-N-1}.$ So

$$
\int_{I(N,k)} g(x)dx \le
2d\cdot2^{-2N} + \left(2^{-2N} + \sum\limits_{M>N} 2^{-2M}\cdot 2^{M-N-1} \right)\cdot(1-2d)^2/\varkappa .
$$

Since
$$
\sum\limits_{M>N} 2^{-2M}\cdot 2^{M-N-1} = 2^{-N-1}\sum\limits_{M>N} 2^{-M} = 2^{-2N}/2,
$$
then
$$
\int_{I(N,k)} g(x)dx \le 2^{-2N}\cdot\left(2d + \frac32\cdot(1-2d)^2/\varkappa\right) < 2^{-2N}\cdot\left(2d + \frac3{2\varkappa}\right).
$$
\end{proof}

\begin{defi}
\label{induction-step}
Let $i$ be a non-negative integer and let the functions $f_{N,k}^j$ be already defined for all admissible $(N,k)$ and all non-negative integers $j$ smaller than $i.$

Let us fix for a while a real number $y_0 = y_0(N, k, i) > 2d\cdot2^{-N}.$
Let $g = g_{N, k, y_0}$ be the function from Definition~\ref{minimal-function}. 

If $i=0,$ we set $f_{N,k}^i = g.$

Let $i>0.$
We set $f_{N, k}^{i} = f_{N,k}^{i-1}$ if the function $f_{N,k}^{i-1}$ is stabilized. In other case
\begin{equation}
\label{maximum-from-definition}
f_{N,k}^i = \max\left(g, f_{M,l}^{i-1} + 2d\cdot \left(2^{-N} - 2^{-M}\right)\right),
\end{equation}
where the maximum is taken over all stabilized functions $f_{M,l}^{i-1}$ such that $N< M.$

Now return to the choice of the number $y_0.$ We must choose it so that the above construction gives
$$\int f_{N, k}^i(x) dx = S_{N, k}.$$
We call the number $y_0$ {\it a reference value} of the function. We call the function $g$ a {\it base part} of the function $f_{N, k}^i$.

Now we choose a non-stabilized function $f_{M, l}^i$ such that
$$
f_{M, l}^i - f_{N, k}^i \ge 2d\cdot \left(2^{-M} - 2^{-N}\right)
$$
for any $N> M$ and admissible pair $N,$ $k$ at each point where both functions are defined. Among such functions we choose one with minimal $M$. If there are several such functions we choose one with minimal $l$ among them. We set
$$T(M, l) = i.$$

Now we extend the just stabilized function $f_{M, l}^i$ to the whole line by requiring that on rays $x \ge l\cdot2^{-M}$ and $x \le (l-1)\cdot2^{-M}$ the function is linear, its slope equals $-\varkappa$ on the right ray, and $\varkappa$ on the left one.
\end{defi}

\begin{figure}
\center
\includegraphics{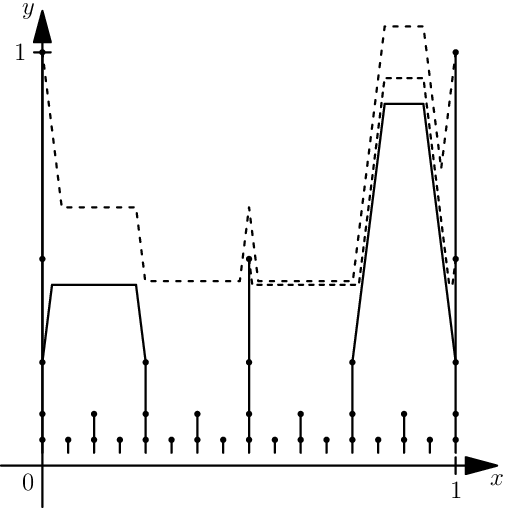}
\includegraphics{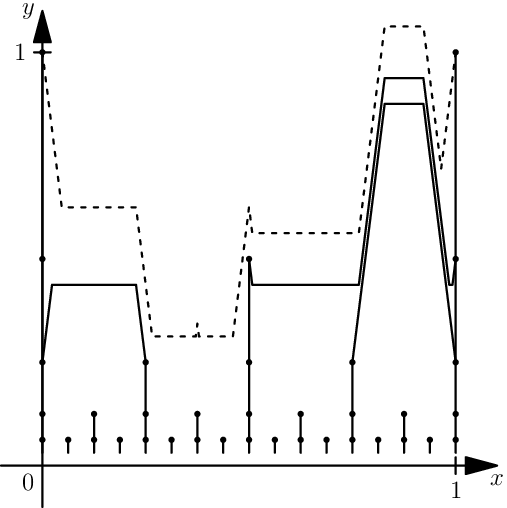}
\caption{A change of functions $f_{N,k}^i$ during stabilization of one of them. Stabilized functions are depicted with solid lines. The dashed line on the interval $(1/2;3/4)$ became higher because one more function appeared in the maximum function~(\ref{maximum-from-definition}). The dashed line on the interval $(1/4;1/2)$ became lower because the integral of the function is preserved.}
\label{induction-step-figure}
\end{figure}

In Figure~\ref{induction-step-figure} we show how functions $f_{N,k}^i$ change during stabilization of one function.

In Claims~\ref{derivative-claim}--\ref{T-is-bijection} we will check correctness of Definition~\ref{induction-step}.

We fix $i$ and we assume that all functions $f_{N,k}^j$ are correctly defined for all non-negative integers $j < i.$ We also assume that for each such $j$ we successfully chose a function whose moment of stabilization equals $j$. We will consider a family of functions $f_{N,k}^i$ depending on the parameter $y_0$ --- the reference value of a function.

\begin{clai}\label{derivative-claim}
Each function $f_{N,k}^j$ is piecewise linear and its derivative equals $0$ or $\pm\varkappa.$
\end{clai}

\begin{proof}
We use induction on $j$. The function was defined as a maximum of a finite number of functions defined on the segment $I(N,k)$ for which the statement is true by induction.
The base of induction for $j=0$ is true by Claim~\ref{piecewise-linear-g}.
After stabilization the function is extended outside the segment $I(N,k)$ as a linear function with a slope $\pm\varkappa.$
\end{proof}

\begin{clai}
\label{boundary-values}
For any $j\le i$ and admissible pair $N,k$
$$
f_{N,k}^j \left((k-1)\cdot2^{-N}\right) = f_{N,k}^j \left(k\cdot2^{-N}\right) = 2^{-N}.
$$
\end{clai}

\begin{proof}
We use induction by $j$. For the base function the equality is true by definition, so the induction base is done. 

If functions $f_{N,k}^{j-1}$ and $f_{N,k}^j$ have distinct values at  $x_0=k\cdot2^{-N}$, then by~(\ref{maximum-from-definition}) $f_{N,k}^j(x_0)= f_{M,l}^{j-1}(x_0) + 2d\cdot \left(2^{-N}-2^{-M}\right)>2^{-N}$ for some stabilized function $f_{M,l}^{j-1}$ and $M>N$. The point $x_0$ can not be the inner point of the segment $I(M,l)$. If $x_0$ is its endpoint, then by induction $f_{M,l}^{j-1}(x_0)=2^{-M}$, hence $ f_{M,l}^{j-1}(x_0) + 2d\cdot \left(2^{-N}-2^{-M}\right)<2^{-N}$, a contradiction. If it is outside the segment, then  $f_{M,l}^{j-1}(x_0)$ is even smaller due to the way we extend stabilized functions. The argument for the second endpoint $x_1=(k-1)\cdot2^{-N}$ is similar, so we have proved the induction step.
\end{proof}

\begin{clai}
\label{influence-of-stabilized2}
Suppose that $f_{M, l}^{j-1}$ is stabilized and $f_{N,k}^{j-1}$  is not stabilized for some $j\le i$ and $M > N.$ Let $x\in I(N,k)$ and let
$$
f_{N,k}^j (x) =  f_{M,l}^{j-1}(x) + 2d\cdot \left(2^{-N} - 2^{-M}\right).
$$
Then $I(M,l)\subset I(N,k)$. If $x$ does not lie in the interior of the segment $I(M,l),$ then $f_{N,k}^j(x) = g(x),$ where $g$ is the base part of the function $f_{N,k}^j(x).$
\end{clai}

\begin{proof}
Suppose the contrary. Let the segment $I(M,l)$ be not contained in the segment $I(N,k).$ Then either  $l\cdot 2^{-M} \le (k-1)\cdot2^{-N}$ or $(l-1)\cdot2^{-M} \ge k\cdot 2^{-N}.$ Consider the first case. The second is similar.

By Claim~\ref{boundary-values} we have $f_{M,l}^{j-1} (l\cdot 2^{-M}) = 2^{-M}.$ Since  $f_{M,l}^{j-1}$ is stabilized, $f_{M,l}^{j-1} (l\cdot 2^{-M} + x) = 2^{-M} - \varkappa x$ for $x\ge0$. In particular, $f_{M,l}^{j-1} ((k-1)\cdot2^{-N})\le 2^{-M}.$ The function $g$ equals $2^{-N}$ at the point $(k-1)\cdot2^{-N}$ and its derivative is at least $-\varkappa$ almost everywhere in the domain of the function. Hence $f_{M,l}^{j-1}  + 2^{-N} - 2^{-M}\le g$ on the segment $I(N,k).$ Since $d<1/2,$ $f_{M,l}^{j-1}  + 2d\cdot\left(2^{-N} - 2^{-M}\right)< g.$ Since $g \le f_{N,k}^j,$ we get a contradiction.

Now let $x$ not lie in the interior of the segment $I(M,l).$ We can assume that $x\ge l\cdot2^{-M},$ the second case is similar. Then

$$
f_{N,k}^{j}(x) - f_{M,l}^{j-1}(x)\ge f_{N,k}^{j} (l\cdot 2^{-M}) - f_{M,l}^{j-1}(l\cdot 2^{-M}) \ge 2d\cdot\left(2^{-N} - 2^{-M}\right),
$$
where the first inequality follows from the fact that the derivative of the function $f_{N,k}^{j}$ is at least $-\varkappa$ almost everywhere in $I(N,k)$ by Claim~\ref{derivative-claim}, and that the derivative of a function $f_{M,l}^{j-1}$ equals $-\varkappa$ to the right from the point  $l\cdot 2^{-M}$; and the second inequality follows from the fact that the function $f_{N,k}^j$ was defined as a maximum~(\ref{maximum-from-definition}) with the function $f_{M,l}^{j-1}$ in it. By assumptions of the present claim both these inequalities are actually equalities. So
$$
f_{N,k}^{j}(x) = 2^{-M} + 2d\cdot\left(2^{-N} - 2^{-M}\right) - \varkappa\cdot\left(x-l\cdot2^{-M}\right) \le g(x)
$$
by definition of $g.$ Since the function $f_{N,k}^{j}$ is not smaller than its base part at each point, $f_{N,k}^{j}(x) = g(x).$
\end{proof}

\begin{clai}
\label{influence-of-stabilized}
Let $j$ be the moment of stabilization of the function $f_{M,l}^{j+1}$ and the segment $I(M,l)$ be not contained in the segment $I(N,k).$
Then $f_{N, k}^{j}=f_{N, k}^{j+1}.$
\end{clai}
\begin{proof}
If the function $f_{N, k}^{j}$ is stabilized, the statement is true.

If $M\le N$ the sets of stabilized functions in the maximum~(\ref{maximum-from-definition}) for functions $f_{N, k}^{j}$ and $f_{N, k}^{j+1}$ coincide. So in this case the statement is also true.

In other cases for the function $f_{N, k}^{j+1}$ this set contains one more function, actually it is the function $f_{M,l}^{j}.$ However by Claim~\ref{influence-of-stabilized2} the function $f_{M,l}^{j}$ does not alter the maximum for the function $f_{N, k}^{j+1}$. This means that $f_{N, k}^{j}=f_{N, k}^{j+1}.$
\end{proof}

Now we discuss how the area below the graph of a function $f_{N,k}^i$ depends on the choice of the reference value $y_0.$

\begin{clai}
\label{maximal-y0-claim}
Let $y_0 = (\varkappa/2 +1)\cdot 2^{-N}.$ Then the function $f_{N,k}^i$ is linear and increasing on the left part of the segment $I(N,k)$, and it is linear and decreasing on the right part.
\end{clai}

\begin{proof}
It is easy to see that the base part $g$ of the function $f_{N,k}^i$ is precisely as described in the statement of the present claim. Let us show that $f_{N,k}^i \equiv g.$ At the endpoints of the segment $I(N,k)$ this equality is true by Claim~\ref{boundary-values}. From equality~(\ref{maximum-from-definition}) $f_{N,k}^i \ge g$ on the whole segment. If at some point the inequality is strict, the absolute value of the derivative at some point is greater than $\varkappa,$ which contradicts Claim~\ref{derivative-claim}.
\end{proof}

\begin{clai}
\label{maximal-y0-integral}
The area below the graph of the function from Claim~\ref{maximal-y0-claim} is greater than $S_{N,k}.$
\end{clai}

\begin{proof}
The area below the graph of such function is a sum of the area of a square with the side $2^{-N}$ and the area of a triangle based on the side of the square. Then
$$
\int f_{N,k}^i(x)dx = 2^{-2N} + \frac12\cdot 2^{-N}\cdot (y_0 - 2^{-N}) = 2^{-2N} + \frac14\varkappa\cdot 2^{-2N} =\left(1 + \frac{\varkappa}4\right)\cdot 2^{-2N}> S_{N,k},
$$
since $S_{N, k} \le 2C\cdot 2^{-2N}$ by inequality~(\ref{S-inequality}) and $1 + \frac{\varkappa}4 > 2C$.
\end{proof}

\begin{clai}
\label{monotonicity-claim1}
The function $f_{N,k}^i$ depends continuously and monotonically on the reference value.
\end{clai}

\begin{proof}
The function $f_{N,k}^i$ was defined as the maximum of the base part, which is continuous and monotonic on $y_0$, and other functions which do not depend on the reference value. Since the maximum function is continuous and monotonic on each of its arguments, we are done.
\end{proof}

\begin{clai}
\label{monotonicity-claim2}
The area below the graph of a function $f_{N,k}^i$ depends continuously and monotonically on the reference value.
\end{clai}

\begin{proof}
This immediately follows from Claim~\ref{monotonicity-claim1} and from monotonicity of the integral.
\end{proof}

\begin{clai}
\label{domination-of-stabilized}
Let $j<i,$ $f_{M_1, l_1}^j$ be stabilized and $I(M_1,l_1)\supset I(M,l)$.
Then on the segment $I(M,l)$
\begin{equation}
\label{domination-inequality}
f_{M_1, l_1}^{j} - f_{M, l}^{j} \ge 2d\cdot \left(2^{-M_1} - 2^{-M}\right).
\end{equation}
\end{clai}

\begin{proof}
Let $j_1$ be the moment of stabilization of the function $f_{M_1, l_1}^j.$ According to the choice of a function to stabilize in Definition~\ref{induction-step}, inequality~(\ref{domination-inequality}) is true for $j=j_1$. We prove the claim by induction on $j$. Let inequality~(\ref{domination-inequality}) be true for smaller $j$ such that $j\ge j_1$ and for all such admissible pairs $(M,l)$ that $I(M_1,l_1)\supset I(M,l).$

Let $j>j_1.$ Then $f_{M_1, l_1}^{j-1} = f_{M_1, l_1}^{j}$. Let the moment of stabilization of the function $f_{M_2, l_2}^{j}$ be $j-1$. If the segment $I(M_2,l_2)$ is not contained in the segment $I(M,l)$, then by Claim~\ref{influence-of-stabilized} $f_{M, l}^{j-1}=f_{M, l}^{j}$ and the inequality is proved.

Let $I(M_2,l_2)\subset I(M,l).$ Then the number of functions in the maximum~(\ref{maximum-from-definition}) for $f_{M, l}^{j}$ is greater by one than such number for the function $f_{M, l}^{j-1}$ (the additional function is $f_{M_2, l_2}^{j-1}$). By Claim~\ref{monotonicity-claim2} the reference value did not increase. Therefore the value of $f_{M, l}^{j}$ can be greater than the value of $f_{M, l}^{j-1}$ only at those points where $f_{M, l}^{j} = f_{M_2, l_2}^{j-1} + 2d\cdot \left(2^{-M} - 2^{-M_2}\right).$ So it is sufficient to check the inequality at such points, let $x_0$ be one of them. Then

\begin{multline*}
f_{M_1, l_1}^{j}(x_0) - f_{M, l}^{j}(x_0) = f_{M_1, l_1}^{j}(x_0) - f_{M_2, l_2}^{j-1}(x_0) - 2d\cdot \left(2^{-M} - 2^{-M_2}\right) = \\
= f_{M_1, l_1}^{j-1}(x_0) - f_{M_2, l_2}^{j-1}(x_0) - 2d\cdot \left(2^{-M} - 2^{-M_2}\right).
\end{multline*}

Since $I(M_2,l_2)\subset I(M_1,l_1),$ we can use the induction assumption:

$$
f_{M_1, l_1}^{j}(x_0) - f_{M, l}^{j}(x_0) \ge 2d\cdot \left(2^{-M_1} - 2^{-M_2}\right) - 2d\cdot \left(2^{-M} - 2^{-M_2}\right) = 2d\cdot \left(2^{-M_1} - 2^{-M}\right).
$$

\end{proof}

\begin{clai}
\label{domination-of-stabilized2}
Let $j<i,$ $f_{M_1, l_1}^j$ and $f_{M_2, l_2}^j$ be stabilized functions and $I(M_1,l_1)\supset I(M_2,l_2)$.
Then
\begin{equation}
\label{domination-inequality2}
f_{M_1, l_1}^{j} - f_{M_2, l_2}^{j} \ge 2d\cdot \left(2^{-M_1} - 2^{-M_2}\right)
\end{equation}
everywhere.
\end{clai}

\begin{proof}
Inequality~(\ref{domination-inequality2}) is true on the segment $I(M_2,l_2)$ by Claim~\ref{domination-of-stabilized}. In particular, it is true at the right endpoint of the segment $I(M_2,l_2)$. To the right from this endpoint it is also true because the function $f_{M_2, l_2}^j$ is linear and its slope equals $-\varkappa,$ and the function $f_{M_1, l_1}^j$ almost everywhere on this ray has derivative at least $-\varkappa.$ Points to the left from the segment $I(M_2, l_2)$ can be considered in a similar way.
\end{proof}

\begin{clai}
\label{clamp-claim}
Let $f_{M_1, l_1}^{i-1}$ and $f_{M_2, l_2}^{i-1}$ be stabilized, $x_0\in I(N,k)$ and $f_{N,k}^i(x_0) = f_{M_2,l_2}^{i-1}(x_0) + 2d\cdot\left(2^{-N} - 2^{-M_2}\right).$ Let also $I(N,k)\supset I(M_1,l_1)\supset I(M_2,l_2)$. Then 
$$f_{N,k}^i(x_0) = f_{M_1,l_1}^{i-1}(x_0) + 2d\cdot\left(2^{-N} - 2^{-M_1}\right).$$
\end{clai}

\begin{proof}
By Claim~\ref{domination-of-stabilized2} 
\begin{equation}
\label{clamp-inequality1}
f_{M_1, l_1}^{i-1}(x_0) - f_{M_2, l_2}^{i-1}(x_0) \ge 2d\cdot \left(2^{-M_1} - 2^{-M_2}\right).
\end{equation}

Since $I(N,k)\supset I(M_1,l_1)\supset I(M_2,l_2)$,  we have $N \le M_1 \le M_2.$ If some of these inequalities are actually equalities, then the statement of the claim is trivially true. So we may assume that $N < M_1 < M_2.$

If $f_{N,k}^{i-1}$ is not stabilized, by Definition~\ref{induction-step}
\begin{equation}
\label{clamp-inequality2}
f_{N, k}^{i}(x_0) - f_{M_1, l_1}^{i-1}(x_0) \ge 2d\cdot \left(2^{-N} - 2^{-M_1}\right).
\end{equation}

We have $f_{N,k}^i(x_0) = f_{M_2,l_2}^{i-1}(x_0) + 2d\cdot(2^{-N} - 2^{-M_2}).$ Therefore inequalities~(\ref{clamp-inequality1}) and~(\ref{clamp-inequality2}) are actually equalities and we are done.

Let $f_{N,k}^{i-1}$ be stabilized. Then inequality~(\ref{clamp-inequality2}) is also true by Claim~\ref{domination-of-stabilized2} and by $f_{N,k}^{i}=f_{N,k}^{i-1}.$ The rest of the proof is similar.
\end{proof}

\begin{clai}
\label{overview-claim}
Let the function $f_{N,k}^{i-1}$ be not stabilized yet and let $g$ be the base part of the function $f_{N,k}^i$. Then there exists a set $\mathcal F$ of stabilized functions such that
\begin{enumerate}[label = \arabic*), ref = \arabic*)]
\item if $f_{M,l}^{i-1}\in\mathcal F,$ then $M>N$ and $I(M,l)\subset I(N,k)$;
\item if $f_{M_1, l_1}^{i-1}\in \mathcal F$ and $f_{M_2,l_2}^{i-1}\in \mathcal F,$ then the interiors of the segments $I(M_1, l_1)$ and $I(M_2, l_2)$ do not intersect;
\item $f_{N,k}^i = \max\limits_{f_{M,l}^{i-1}\in\mathcal F}\left(g, f_{M,l}^{i-1} + 2d\cdot\left(2^{-N}-2^{-M}\right)\right).$
\end{enumerate}
\end{clai}

\begin{proof}
Let $\mathcal F_0$ be the set of all stabilized functions $f_{M,l}^{i-1}$, where $M>N,$ for which there exists a point $x\in I(N,k)$ such that
\begin{equation}
\label{anchor-equality}
f_{N,k}^i (x) =  f_{M,l}^{i-1}(x) + 2d\cdot \left(2^{-N} - 2^{-M}\right).
\end{equation}

The set of such points $x$ we denote by $X$.

If $f_{M,l}^{i-1}\in \mathcal F_0,$ then $I(M,l)\subset I(N,k)$ by Claim~\ref{influence-of-stabilized2}.

If $f_{M_1,l_1}^{i-1}\in \mathcal F_0,$ $f_{M_2,l_2}^{i-1}\in \mathcal F_0$ and $I(M_1, l_1)\supset I(M_2, l_2),$ we remove the function $f_{M_2,l_2}^{i-1}$ from the set $\mathcal F_0$. By Claim~\ref{clamp-claim}, if the equality~(\ref{anchor-equality}) is true at some point $x$ for $M = M_2$ and $l = l_2$, then it is also true at the same point for $M = M_1$ and $l = l_1.$ Hence for any point $x\in X$ there exists a function from the reduced set $\mathcal F_0$ such that the equality~(\ref{anchor-equality}) is true. We continue to remove functions from $\mathcal F_0$ while it is possible. The obtained set of functions is the required set $\mathcal F.$
\end{proof}

\begin{clai}
\label{minimal-y0}
If the reference value is not greater than $\left(2d + C^{-1}/4\right)\cdot 2^{-N}$, the area below the graph of the function $f_{N,k}^i$ is less than $S_{N,k}.$
\end{clai}

\begin{proof}
We will use notations from Claim~\ref{overview-claim}. We have
$$
\mathcal F = \{f_{M_1,l_1}^{i-1}, f_{M_2, l_2}^{i-1}, \dots, f_{M_n, l_n}^{i-1}\},
$$ where $I(M_j,l_j) = [a_j; b_j]$ and
$$
(k-1)\cdot2^{-N} \le a_1 < b_1 \le a_2 < b_2 \le \dots \le a_n < b_n\le k\cdot2^{-N}.
$$
Also
\begin{equation}
\label{f-is-maximum}
f_{N,k}^i = \max\limits_{j = 1}^n\left(g, f_{M_j,l_j}^{i-1} + 2d\cdot\left(2^{-N}-2^{-M_j}\right)\right).
\end{equation}

Let $f_j,$ $j=1,\dots,n$, be such functions on the segment $I(N,k)$ that $f_j(x) = f_{M_j,l_j}^{i-1}(x)$ for $x\in I(M_j,l_j)$ and $f_j(x) = 0$ for the rest $x$. By Claim~\ref{influence-of-stabilized2}

$$
f_{N,k}^i = \max\limits_{j = 1}^n\left(g, f_j + 2d\cdot\left(2^{-N}-2^{-M_j}\right)\right).
$$

Since
$$
f_{N,k}^i \le g + \max\limits_{j = 1}^n\left(f_j + 2d\cdot\left(2^{-N}-2^{-M_j}\right)\right)
$$
and
$$\max\limits_{j = 1}^n\left(f_j + 2d\cdot\left(2^{-N}-2^{-M_j}\right)\right) < \max\limits_{j=1}^n\left(f_j\right) + 2d\cdot2^{-N}, 
$$
then
\begin{equation}
\label{integral-bound}
\int_{I(N,k)}f_{N,k}^i(x)dx < \int_{I(N,k)} g(x)dx + \int_{I(N,k)} \max\limits_{j=1}^n\left( f_j(x)\right) dx + 2d\cdot2^{-2N}.
\end{equation}

We estimate the first and the second summands on the right side of inequality~(\ref{integral-bound}).

Let the reference value equal $y_0 = 2d\cdot2^{-N} + \varepsilon$ and be less than $2^{-N}.$ By Definition~\ref{minimal-function}
$$
g \le 
\varepsilon + 
\max\left(2d\cdot2^{-N},  \max\limits_{l\cdot2^{-M}\in I(N,k)}\left(2^{-M} + 2d\cdot\left(2^{-N} - 2^{-M}\right) -h_{l\cdot2^{-M}}\right)\right).
$$

Then by Claim~\ref{g-integral}
$$
\int_{I(N,k)} g(x)dx
<
\varepsilon \cdot2^{-N} + 
2^{-2N}\cdot\left(2d + \frac3{2\varkappa}\right).
$$

The equation 
$$
\max\limits_{j=1}^n\left(f_j\right)
=
\sum\limits_{j=1}^n f_j
$$
is true almost everywhere.

By definition
$$
\int_{\mathbb R} f_j dx = \int\limits_{a_j}^{b_j} f_{M_j,l_j}^{i-1}(x) dx.
$$

By Definition~\ref{induction-step}
$$
\int\limits_{a_j}^{b_j} f_{M_j,l_j}^{i-1}(x) dx = S_{M_j, l_j}.
$$

We substitute the obtained equalities and inequalities into inequality~(\ref{integral-bound}) and we obtain
$$
\int_{I(N,k)}f_{N,k}^i(x)dx <\varepsilon\cdot2^{-N} + 2^{-2N}\cdot\left(2d + \frac3{2\varkappa}\right) + \sum\limits_{j=1}^n S_{M_j, l_j} + 2d\cdot2^{-2N}.
$$

By definition $S_{N,k}\ge A_{N,k} + \sum\limits_{j=1}^n S_{M_j, l_j}.$ By assumptions of Lemma~\ref{main-lemma} $A_{N,k}\ge C^{-1}\cdot2^{-2N}.$ Therefore
$$
\int_{I(N,k)}f_{N,k}^i(x)dx < S_{N,k} + 2^{-2N}\cdot\left(-C^{-1} + 4d + \frac3{2\varkappa}\right) + \varepsilon\cdot2^{-N}.
$$

Since $4d < C^{-1}/2$ and $3/2\varkappa < 3C^{-1}/16$, then
$$
\int_{I(N,k)}f_{N,k}^i(x)dx < S_{N,k} - 2^{-2N}\cdot 5C^{-1}/16 + \varepsilon\cdot 2^{-N}.
$$

If $\varepsilon \le 2^{-N}\cdot 5C^{-1}/16,$ then

$$
\int_{I(N,k)}f_{N,k}^i(x)dx < S_{N,k}.
$$
\end{proof}

\begin{clai}
\label{minimal-y0-2}
If the reference value can be decreased so that the area below the graph of a function $f_{N,k}^i$ does not change, this area is less than $S_{N,k}.$
\end{clai}

\begin{proof}
We note that the reference value of the function $f_{N,k}^i$ is not greater than $2^{-N}.$ Suppose the contrary. For $x_0 = \left(k-\frac12\right)\cdot2^{-N}$ we have $g(x_0)>2^{-N}$, hence by equality~(\ref{f-is-maximum}) $f_{N,k}^i(x_0)>2^{-N}.$ But $f_{M_j,l_j}^{i-1}(x_0) + 2d\cdot(2^{-N}-2^{-M_j}) < 2^{-N}$ for any $j=1,\dots,n.$ This means that the reference value can be decreased so that the integral decreases, a contradiction.

If the reference value can be decreased so that the area below the graph of the function $f_{N,k}^i$ does not change, by Claim~\ref{monotonicity-claim2} the function $f_{N,k}^i$ itself does not change. 
This means that the function will not change if we further decrease the reference value. Hence we can assume that the reference value is as close as we want to $2d\cdot2^{-N}$.
It remains to use Claim~\ref{minimal-y0}.
\end{proof}

From Claims~\ref{maximal-y0-integral},~\ref{monotonicity-claim2},~\ref{minimal-y0} and~\ref{minimal-y0-2} it follows that in Definition~\ref{induction-step} there exists the unique such reference value $y_0$ that
$$
\int_{I(N,k)}f_{N,k}^i(x)dx = S_{N,k}.
$$

So for now we have correctly defined all functions $f_{N,k}^i.$ We further check correctness of Definition~\ref{induction-step}. We show that one of these functions satisfies all conditions for stabilization.

\begin{clai}
\label{thin-consequence}
Let $N_1 < N_2,$ $x_0\in I(N_1,k_1)\cap I(N_2,k_2)$ and
\begin{equation}
\label{thin-inequality}
f_{N_1,k_1}^i(x_0) - f_{N_2, k_2}^i(x_0) < 2d\cdot\left(2^{-N_1} - 2^{-N_2}\right).
\end{equation}

Then the functions $f_{N_1,k_1}^i$ and $f_{N_2, k_2}^i$ are not stabilized, the function $f_{N_2, k_2}^i$ equals its base part at the point $x_0$ and
$$
y_1 - y_2 < 2d\cdot\left(2^{-N_1} - 2^{-N_2}\right),
$$
where $y_1$ and $y_2$ are reference values of $f_{N_1,k_1}^i$ and $f_{N_2, k_2}^i.$
\end{clai}

\begin{proof}
If the interiors of the segments $I(N_1,k_1)$ and $I(N_2,k_2)$ do not intersect, $x_0$ is their common endpoint. Then
$$
f_{N_1,k_1}^i(x_0) - f_{N_2, k_2}^i(x_0) = 2^{-N_1} - 2^{-N_2} > 2d\cdot\left(2^{-N_1} - 2^{-N_2}\right),
$$
since $d<1/2,$ a contradiction.

Hence $I(N_1,k_1)\supset I(N_2,k_2).$ By Claim~\ref{domination-of-stabilized} and by inequality~(\ref{thin-inequality}) the function $f_{N_1,k_1}^i$ is not stabilized yet. Let the function $f_{N_2,k_2}^i$ be stabilized. Since the function $f_{N_1,k_1}^i$ was defined as the maximum in~(\ref{maximum-from-definition}) with an expression $f_{N_2,k_2}^i + 2d\cdot\left(2^{-N_1} - 2^{-N_2}\right)$ in it, we have a contradiction with inequality~(\ref{thin-inequality}).

Hence both functions $f_{N_1,k_1}^i$ and $f_{N_2, k_2}^i$ are not stabilized yet. We denote their base parts by $g_1$ and $g_2$ respectively. Let $f_{N_2, k_2}^i (x_0) > g_2(x_0).$ Then there exists $M>N_2$ and a stabilized function $f_{M,l}^{i-1}$ such that
$$
 f_{N_2, k_2}^i (x_0) = f_{M,l}^{i-1}(x_0) + 2d\cdot\left( 2^{-N_2} - 2^{-M} \right).
$$

Then
$$
f_{N_1, k_1}^i(x_0) - f_{M,l}^{i-1}(x_0) = f_{N_1, k_1}^i(x_0) - f_{N_2, k_2}^i(x_0) + 2d\cdot\left( 2^{-N_2} - 2^{-M} \right) < 2d\cdot\left(2^{-N_1} - 2^{-M}\right)
$$
by inequality~(\ref{thin-inequality}). This is impossible because the function $f_{N_1,k_1}^i$ was defined as the maximum in~(\ref{maximum-from-definition}) with an expression $f_{M,l}^{i-1} + 2d\cdot\left(2^{-N_1} - 2^{-M}\right)$ in it. So we have a contradiction, hence $f_{N_2, k_2}^i (x_0) = g_2(x_0).$

Since $f_{N_1,k_1}^i\ge g_1,$ by inequality~(\ref{thin-inequality})
\begin{equation}
\label{eq1}
g_1(x_0) - g_2(x_0) < 2d\cdot\left(2^{-N_1}-2^{-N_2}\right).
\end{equation}
We consider four cases.

\smallskip
\noindent{\bf Case $y_1 \ge 2^{-N_1}$ and $y_2 \ge 2^{-N_2}$.}
We recall that in this case the functions $g_1$ and $g_2$ have the form as in Figure~\ref{aux-functions} on the left. Since $N_1 < N_2$, $d<1/2$ and functions have same slope, inequality~(\ref{eq1}) can be satisfied only if $g_1(x_0) = y_1.$ Since $g_2(x_0)\le y_2,$ by inequality~(\ref{eq1}) we have
$$
y_1 - y_2 < 2d\cdot\left(2^{-N_1} - 2^{-N_2}\right).
$$

\smallskip
\noindent{\bf Case $y_1<2^{-N_1}$ and $y_2 \ge 2^{-N_2}$.}
Then $g_1(x_0) \ge y_1$ and $g_2(x_0) \le y_2.$ Therefore from inequality~(\ref{eq1})

$$
y_1 - y_2 \le g_1(x_0) - g_2(x_0) < 2d\cdot\left(2^{-N_1}-2^{-N_2}\right).
$$

\smallskip
\noindent{\bf Case $y_1<2^{-N_1}$ and $y_2<2^{-N_2}$.}

Then by definition for $n=1,2$
\begin{multline*}
g_n = \max\left( y_n, \max\limits_{l\cdot2^{-M}\in I(N_n,k_n)}\left( 2^{-M} + 2d\cdot\left(2^{-N_n} - 2^{-M}\right) - h_{l\cdot2^{-M}} \right) \right) = \\
= 2d\cdot2^{-N_n} + \max\left( y_1 - 2d\cdot2^{-N_n}, \max\limits_{l\cdot2^{-M}\in I(N_n,k_n)}\left( 2^{-M}\cdot(1-2d) - h_{l\cdot2^{-M}} \right) \right)
\end{multline*}
and
\begin{multline*}
g_1 - g_2 = 2d\cdot\left(2^{-N_1}-2^{-N_2}\right) + \max\left( y_1 - 2d\cdot2^{-N_1}, \max\limits_{l\cdot2^{-M}\in I(N_1,k_1)}\left( 2^{-M}\cdot(1-2d) - h_{l\cdot2^{-M}} \right) \right) - \\
- \max\left( y_2 - 2d\cdot2^{-N_2}, \max\limits_{l\cdot2^{-M}\in I(N_2,k_2)}\left( 2^{-M}\cdot(1-2d) - h_{l\cdot2^{-M}} \right) \right).
\end{multline*}

Since
$$
\max\limits_{l\cdot2^{-M}\in I(N_1,k_1)}\left( 2^{-M}\cdot(1-2d) - h_{l\cdot2^{-M}} \right) \ge
\max\limits_{l\cdot2^{-M}\in I(N_2,k_2)}\left( 2^{-M}\cdot(1-2d) - h_{l\cdot2^{-M}} \right)
$$
(the maximum on the left is taken over the bigger set of functions), then by inequality~(\ref{eq1})
$$
\max\left( y_2 - 2d\cdot2^{-N_2}, \max\limits_{l\cdot2^{-M}\in I(N_2,k_2)}\left( 2^{-M}\cdot(1-2d) - h_{l\cdot2^{-M}}(x_0) \right) \right) = y_2 - 2d\cdot2^{-N_2}.
$$

Therefore $g_2(x_0) = y_2.$ Since $g_1(x_0)\ge y_1,$
$$
y_1 - y_2 \le g_1(x_0) - g_2(x_0) < 2d\cdot\left(2^{-N_1}-2^{-N_2}\right).
$$

\smallskip
\noindent{\bf Case $y_1 \ge 2^{-N_1}$ and $y_2 < 2^{-N_2}$.} Then $g_1 \ge 2^{-N_1}$ and $g_2 \le 2^{-N_2},$ which contradicts inequality~(\ref{eq1}), since $N_1 < N_2$ and $d<1/2.$
\end{proof}

\begin{clai}
\label{existence-of-stabilizable}
Let $y_0$ be a reference value of a not stabilized function $f_{N_0,k_0}^i.$ Then there exist an integer $M$ satisfying
$$
M \le \max\left(N_0, \log_2\frac{\varkappa/2 + 1}{y_0 - 2d\cdot2^{-N_0}}\right)
$$
and a function $f_{M,l}^i$ which satisfies all conditions for stabilization.
\end{clai}

\begin{proof}
Let
$$
y_0 = 2d\cdot2^{-N_0} + \varepsilon.
$$ 
By Definition~\ref{induction-step} $\varepsilon>0.$

If the function $f_{N_0,k_0}^i$ satisfies all conditions for stabilization there is nothing to prove. Suppose the contrary. Then there exists $N_1>N_0$ and a function $f_{N_1,k_1}^i$ defined at some point $x_0\in I(N_0, k_0)$ such that $f_{N_0,k_0}^i(x_0) - f_{N_1,k_1}^i(x_0) < 2d\cdot\left(2^{-N_0} - 2^{-N_1}\right).$ By Claim~\ref{thin-consequence} the function $f_{N_1,k_1}^i$ is not stabilized yet. Let $y_1$ be its reference value. Then by Claim~\ref{thin-consequence}
$$
y_0-y_1 < 2d\cdot\left(2^{-N_0} - 2^{-N_1}\right).
$$

By similar argument, we can construct a sequence $N_0 < N_1 < \dots < N_m$ and a sequence of not stabilized functions $f_{N_j,k_j}^i$ such that
$$
y_{j-1} - y_j < 2d \cdot\left(2^{-N_{j-1}} - 2^{-N_j}\right).
$$

We sum up all inequalities and we get
$$
y_0 - y_m < 2d\cdot\left(2^{-N_{0}} - 2^{-N_m}\right).
$$

Since $y_0 = 2d\cdot2^{-N_0} + \varepsilon,$ 
$$
y_m > \varepsilon + 2d\cdot 2^{-N_m} = \left(\varepsilon\cdot 2^{N_m} + 2d\right)\cdot2^{-N_m}.
$$

If $\varepsilon\cdot2^{N_m} > \varkappa/2 + 1,$ we have a contradiction with Claims~\ref{maximal-y0-claim} and~\ref{maximal-y0-integral}.
\end{proof}

By Claim~\ref{existence-of-stabilizable} on each step $i$ we can choose a function to stabilize. To show correctness of Definition~\ref{induction-step} it remains to prove that the mapping $T$ is defined on all admissible pairs.

\begin{clai}
\label{T-is-bijection}
Any function $f_{N,k}^i$ stabilizes at some moment.
\end{clai}

\begin{proof}
We fix some non-negative integer $i$.

We choose a not stabilized function $f_{N,k}^i$ with minimal $N$, and with minimal $k$ among such functions. By Claim~\ref{minimal-y0} its reference value is greater than $\left(2d + C^{-1}/4\right)\cdot 2^{-N}.$ Then by Claim~\ref{existence-of-stabilizable} there exists such $M$ that
$$
M \le N + \log_2\left(4C\cdot(\varkappa/2 + 1)\right),
$$
and there exists a function $f_{M,l}^i$ which satisfies all conditions for stabilization.

Since the number of functions $f_{M,l}^i$ with such $M$ is finite, the function $f_{N,k}^i$ stabilizes at some moment.
\end{proof}

\begin{proof}[Proof of Lemma~\ref{main-lemma}]
Let $f_{N,k}^i$ be functions from Definition~\ref{induction-step}.

For each admissible pair $(N,k)$ we define a function on a segment $I(N,k)$

$$
f_{N,k} = \lim\limits_{i\to+\infty} f_{N,k}^i.
$$

The limit exists because $f_{N,k}^i = f_{N,k}^{i+1}$ if $i\ge T(N,k).$

The conditions of the lemma are satisfied for functions $f_{N,k}$ because similar conditions are satisfied for functions $f_{N,k}^i$: conditions 1) and 5) by definition, condition 2) by Claim~\ref{derivative-claim}, condition 3) by Claim~\ref{boundary-values}, condition 4) by Claim~\ref{domination-of-stabilized2}.
\end{proof}

\section{Standard tubular neighborhoods of Legendrian Lavrentiev knots}
\label{contact-section}

We recall that a rectifiable curve is called {\it Lavrentiev} if its arc-length parametrization is biLipschitz. A Lavrentiev link is called {\it Legendrian} if the integral of the contact form on any its subarc is zero.

To work with piecewise smooth links, spatial graphs, triangulations and 2-complexes obtained by gluing embedded surfaces it is convenient to have non-smooth homeomorphisms. Contact topology is not an exception. Here we develop an important tool called the standard tubular neighborhood for Legendrian Lavrentiev links. See~\cite{Pra, Pra2} for related results.

A (locally) biLipschitz homeomorphism is called {\it contact} if it preserves the set of Legendrian Lavrentiev links.

\begin{proof}[Proof of Corollary~\ref{contact-corollary}]
By \cite[Proposition 5.26]{Pra} there exists a neighborhood $U$ of $L = \iota(\mathbb S^1),$ a smoothly embedded surface $S$ in $U$ and a contact diffeomorphism $\Psi:S\times\mathbb R\to U$ such that
\begin{itemize}
\item
the projection $\mathrm{pr}:U\to S$ is injective on $L$;
\item
the contact structure in $U$ is $\ker(dz + \mathrm{pr}^*\beta),$ where $z:U\to\mathbb R$ is the projection and $\beta$ is a 1-form on $S$.
\end{itemize}

We can assume that $S$ is an open annulus equipped with a smooth coordinate system $(x, y)$ on it such that $x=0$ is a core of the annulus and the annulus is given by $x_0< x<x_1,$ $y\in\mathbb R/\mathbb Z.$

Since $dz + \mathrm{pr}^*\beta$ is a contact form, $d\beta$ is an area form on $S.$  We can assume that $(x,y)$ is a positive coordinate system. The following lemma is standard and similar to Darboux theorem.

\begin{lemm}
Let $\beta = \beta_x dx + \beta_y dy$ and $d\beta = fdx\wedge dy,$ where $\beta_x, \beta_y,$ and $f$ are smooth functions.
Then there exists a function $g$ on the annulus such that the following is a new coordinate system
$$\widetilde x\left( x\right) = \int_0^1 \beta_y (x,t) dt, \quad
\widetilde y( x, y) = \int_0^{y} f(x, t) dt / \int_0^1 \frac{\partial \beta_y}{\partial x} (x,t) dt,\quad
\widetilde z(x,y,z) = z - g
$$
and
$$
dz + \mathrm{pr}^*\beta = d\widetilde z + \widetilde xd\widetilde y.
$$
\end{lemm}

\begin{proof}
It is clear that $\widetilde y(x, 0) = 0.$ Also 
$$
\widetilde y(x, 1) = \int_0^1 \left(\frac{\partial \beta_y}{\partial x} - \frac{\partial \beta_x}{\partial y}\right)(x,t) dt / \int_0^1 \frac{\partial \beta_y}{\partial x} (x,t) dt = 1$$
for any $x$ because 
$$
\int_0^1 \frac{\partial \beta_x}{\partial y}(x,t) dt = \beta_x(x,1) - \beta_x(x,0) = 0.
$$
We also see here that $\frac{\partial\widetilde x}{\partial x} = \int_0^1 f(x,t)dt>0.$

Also
$
d\beta = d \widetilde x\wedge d\widetilde y.
$
Thus 
$$
d(\widetilde xd\widetilde y - \beta) = 0.
$$

The integral of $\widetilde xd\widetilde y - \beta$ on any circle $x = \mathrm{const},$ $0\le y\le 1$ is zero. Therefore
$$
\widetilde xd\widetilde y - \beta = dg.
$$
for some function $g$ on the annulus.

It is clear that
$$
dz + \mathrm{pr}^*\beta = d\widetilde z + \widetilde xd\widetilde y.
$$
\end{proof}

So we can assume from the beginning that $\beta = xdy.$
If a circle $x = \mathrm{const}$ does not intersect the curve $\mathrm{pr}(L)$ then they cobound an annulus. By Stokes theorem (see \cite[Chapter IX, Theorems 5A and 7A]{Whi}) the integral of $\beta$ on such circle $x = \mathrm{const}$ is positive if we orient it accordingly. This means that $x_0<0$ and $x_1>0.$ By Theorem~\ref{the-technical-lemma-without-parameters} there is an area-preserving biLipschitz mapping $\Phi:S\to S$ such that
\begin{itemize}
\item
$\Phi$ is the identity near the boundary of $S$;
\item
$\Phi(0, y) = \mathrm{pr}(\iota(y)).$
\end{itemize}

Let $\widetilde \Phi:U\to U$ be the lifting of $\Phi$ as constructed in the introduction. The curve $\widetilde\Phi(\{x = z = 0\})$ can be obtained from $L$ by a parallel translation in the $z$-direction because these two curves are both Legendrian and have the same projection to $S$, so the statement follows by~\cite[Lemma 3.11]{Pra}. By choosing another lift we can assume that $\widetilde\Phi(\{x = z = 0\}) = L.$ So the mapping $\widetilde\Phi\circ\Psi$ is what we need.
\end{proof}

\textsc{Moscow State University}

\textsc{Moscow Center of Fundamental and Applied Mathematics}

{\it E-mail}: \href{mailto: 0x00002a@gmail.com}{\texttt{0x00002a@gmail.com}}

\end{document}